\documentclass[11pt]{amsart}

\usepackage{amsmath, amsthm, amssymb}

\usepackage{palatino}         % text font
\usepackage{comment}

\usepackage[abs]{overpic}		  % overpic automatically loads graphicx
\usepackage{subcaption}
\usepackage{xcolor}
\definecolor{indigo}{rgb}{0.29, 0.0, 0.51}  % custom colors
\definecolor{lblue}{RGB}{0, 174, 239}
\usepackage[colorlinks, urlcolor=indigo, linkcolor=indigo, citecolor=indigo]{hyperref}

\theoremstyle{plain}
\newtheorem{theorem}{Theorem}

\newtheorem{proposition}[theorem]{Proposition}
\newtheorem{lemma}[theorem]{Lemma}

\theoremstyle{definition}

\theoremstyle{remark}
\newtheorem{remark}[theorem]{Remark}
\newtheorem{example}[theorem]{Example}

\numberwithin{theorem}{section}

\newcommand{\dfn}[1]{{\em #1}}        % definition
\newcommand{\R}{\mathbb{R}}           % the real numbers
\newcommand{\Z}{\mathbb{Z}}           % the integers
\makeatletter
\newcommand*\bigcdot{\mathpalette\bigcdot@{0.6}}
\newcommand*\bigcdot@[2]{\mathbin{\vcenter{\hbox{\scalebox{#2}{$\m@th#1\bullet$}}}}}
\makeatother

\DeclareMathOperator\tb{tb}                   % Thurston-Bennequin
\DeclareMathOperator\tbb{\overline {\tb}}     % maximum Thurston-Bennequin
\DeclareMathOperator\rot{rot}                 % rotation
\begin{document}

% title
\title{Cable links of uniformly thick knot types}

% author information
\author{Rima Chatterjee}

\author{John B. Etnyre}

\author{Hyunki Min}

\author{Thomas Rodewald}

\address{Department of Mathematics \\ Ohio State Univerity\\ Columbus, OH}
\email{chatterjee.198@osu.edu, rchattmath@gmail.com}

\address{School of Mathematics \\ Georgia Institute of Technology \\  Atlanta, GA}
\email{etnyre@math.gatech.edu}
\email{tomrodewald@gatech.edu}

\address{Department of Mathematics\\University of California\\LA, California}
\email{hkmin27@math.ucla.edu}
%\subjclass[2020]{57R17}

% abstract
\begin{abstract}
  In this paper, we study Legendrian realizations of cable links of knot types that are uniformly thick but not Legendrian simple, extending prior work of Dalton, the second author, and Traynor. This leads to new phenomena, such as stabilized Legendrian links that are smoothly isotopic and component-wise Legendrian isotopic, but are not Legendrian isotopic. In our study, we establish new results for cable links whose cabling slope is sufficiently negative.
We will also show how to classify Legendrian knots in (most) negative cables of twist knots. This is done by introducing a new technique to the study of cables based on Legendrian surgeries. 
\end{abstract}

\maketitle
%\tableofcontents

%%%%%%%%%%%%%%%%%%%%%%%%%%%%%%%%%%%%
\section{Introduction}
%%%%%%%%%%%%%%%%%%%%%%%%%%%%%%%%%%%%
Legendrian knots have been studied quite a bit \cite{ChenDingLi15, EliashbergFraser09, EtnyreHonda01b, EtnyreHonda03, EtnyreHonda05, EtnyreLafountainTosun12, EtnyreNgVertesi13}, leading to a rich theory with many applications to the construction of contact manifolds and the understanding of their subtle properties. There has been less study of Legendrian links.  The first classification results for Legendrian links were obtained by Ding and Geiges in \cite{DingGeiges10}, where they studied a class of two-component links consisting of an unknot and a cable of the unknot that were Legendrian simple (that is determined by their classical invariants). Geiges and Onaran \cite{GeigesOnaran20, GeigesOnaran20b} studied Legendrian Hopf links in any contact structure on $S^3$, and Geiges, Onaran, and the first author \cite{ChatterjeeGeigesOnaranPre} studied Hopf links in many contact structures on the lens space $L(p,1)$. The first general structure result is due to Dalton, Traynor, and the second author. In \cite{DaltonEtnyreTraynor2024}, they classified Legendrian torus links and Legendrian representatives of cable links of a knot type $K$ that was both uniformly thick and Legendrian simple. We will discuss these terms later, but they impose a strong restriction on the knot types that can be considered. Nonetheless, they were able to show many new results about Legendrian isotopies of such links. 

The main goal of this paper is to extend the results of \cite{DaltonEtnyreTraynor2024} to a wider class of knots. We will be able to extend some of the results to any knot type and other results to uniformly thick knot types, thus removing one of the hypotheses needed in \cite{DaltonEtnyreTraynor2024} to understand Legendrian representatives of cable links. This study will reveal new phenomena about Legendrian isotopy classes of links. For example, a corollary of our work establishes the following new phenomena.

\begin{theorem}\label{newpenom}
There exist pairs of Legendrian links $L$ and $L'$ that are smoothly isotopic, the components of $L$ and $L'$ have been stabilized many times,
and each component of $L$ is Legendrian isotopic to a component of $L'$, but $L$ and $L'$ are not Legendrian isotopic. 
\end{theorem}

\begin{remark}
Previously, Jordan and Traynor \cite{JordanTraynor06} had constructed pairs of links that were smoothly isotopic and component-wise isotopic but not smoothly isotopic. The examples had components that were maximal Thurston-Bennequin invariant Legendrian unknots and were distinguished using generating family techniques. These were the first known examples of this surprising phenomenon, and the only known examples before the ones in this paper. What makes our examples particularly interesting is that the links have been stabiilized multiple times. This forces most invariants, such as Legendrian contact homology or generating function homology, to vanish. So it would seem these examples can only be distinguished via the geometric techniques used in this paper. 
\end{remark}

We will also see the following result.

\begin{theorem}\label{thm:isotopy}
  The components of a maximum Thurston-Bennequin Legendrian realization of a $(np,nq)$-cable of a uniform thick knot type $K$ must all be Legendrian isotopic. If $q/p\geq \lceil w(K)\rceil$, then the same result is true without the assumption of uniform thickness of $K$. (Here $w(K)$ is the width of $K$; see below for a definition.)

  More generally, for a uniformly thick knot type $K$, any two components of a Legendrian representative of the $(np,nq)$-cable of $K$ with the same classical invariants are Legendrian isotopic. 
\end{theorem}

Theorem~\ref{thm:isotopy} follows directly from Theorems~\ref{maingsc}, \ref{maininteger}. and~\ref{negative} below.
We note that  \cite{DaltonEtnyreTraynor2024} showed that there are similar restrictions on the components of a maximum Thurston-Bennequin Legendrian realization of a $(np,nq)$-cable of a Legendrian simple and uniformly thick knot type. In that paper, however, it was shown that all components of the cable must have the same classical invariants, which implies that they are Legendrian isotopic. In Theorem~\ref{thm:isotopy}, we see that all the components are Legendrian isotopic, even without assuming the Legendrian simpleness condition. %though it is not forced by the classical invariants. 

\begin{remark}
Below, we will also give a classification of Legendrian representatives of negative cable knots of twist knots. In \cite{ChakrabortyEtnyreMin2024} a general procedure for classifying negative cable knots was given, but it required extensive knowledge of contact structures on solid tori in that knot type, which is usually difficult to determine. We will show how to use an understanding of Legendrian surgeries on Legendrian knots to understand Legendrian representatives of their cables. To the author's knowledge, this is a new approach to this much studied problem. 
\end{remark}

To state our results, we recall a few definitions. Given a knot type $K$ we define the \dfn{width} $w(K)$ of $K$ to be the supremum of slopes of dividing curves on the boundary of solid tori with a convex boundary in the knot type $K$ (meaning that the core of the solid torus is isotopic to $K$). We call $K$ \dfn{uniformly thick} if any solid torus in the knot type $K$ is contained in a solid torus that is a standard neighborhood of a maximal Thurston-Bennequin representative of $K$. We notice that this implies $w(K)=\tbb(K)$, where $\tbb(K)$ is the maximal Thurston-Bennequin invariant of $K$. 

For relatively prime integers $p$ and $q$, the \dfn{$(p,q)$-cable} of a knot $K$ is the curve $K_{(p,q)}$ on the boundary of tubular neighborhood $N$ of $K$ realizing the homology class $p\lambda+q\mu$ in $H_1(\partial N)$ where $\mu$ is the meridian of $N$ and $\lambda$ is the Seifert longitude of $K$. The $(np,nq)$-cable link $K_{n(p,q)}$ is simply $n$ copies of $K_{(p,q)}$ on $\partial N$. 

Below, when $q/p>\lceil w(K)\rceil$, we will completely classify Legendrian representatives of $(np,nq)$-cable links for {\bf any knot type $K$} in terms of Legendrian representatives of $K$. We call these \dfn{greater-sloped cables}. We are able to recover all the results of \cite{DaltonEtnyreTraynor2024} with no hypothesis on $K$! We will also be able to understand much about such cables when $q/p\leq \tbb(K)$ if $K$ is uniformly thick, thus removing one of the hypotheses needed in \cite{DaltonEtnyreTraynor2024}. When $q/p<\tbb(K)$ we call these \dfn{lesser-sloped cables} and when $q/p=\tbb(K)$ we call these \dfn{$\tbb(K)$-sloped cables}. We note that when considering uniform thick knot types $w(K)=\tbb(K)$; so all possible slopes will be considered for such knot types. To apply our results to some specific knots we will also show the following theorem.

\begin{theorem}\label{thm:uniform}
  The twist knots $K_n$ for $n \notin [-3,2]$ are uniformly thick.
\end{theorem}

The twist knot $K_n$ is shown in Figure~\ref{fig:twk}. The uniform thickness of the positive twist knots with $n>2$ was shown in \cite{George13,GeorgeMyers20}, and it will be shown for the negative twist knots in Section~\ref{sec:uniform}. When $n=0,-1$, the knot $T_n$ is the unknot, and when $n=-2$ we have the right-handed trefoil. These knots are not uniformly thick. When $n=1$, the twist knot is the left-handed trefoil, which is uniformly thick \cite{EtnyreHonda01b}. So the only unknown cases are  $n=2,-3$ where $K_n$ is the figure-eight knot. The existence of non-thickenable tori in the figure-eight knot complement \cite{ConwayMin20} suggests that it is not uniformly thick, but it is not yet known whether these tori sit in the standard contact structure on $S^3$.  

It is well-known \cite{Chekanov02, EtnyreNgVertesi13}, that most of the negative twist knots are not Legendrian simple. We will recall their classification when needed in the examples below. 

\begin{figure}[htb]{
\begin{overpic}%[grid,tics=10] 
{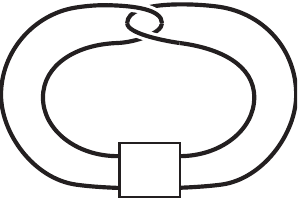}
\put(68, 11){$n$}
\end{overpic}}
\caption{The twist knot $T_n$ where the box contains $n$ positive half=twists if $n>0$ and $|n|$ negative half twists if $n<0$.} 
\label{fig:twk}
\end{figure}

We note that the understanding of cables with slopes $q/p \in (\tbb(K),\lceil w(K)\rceil]$ seems difficult as we know in this range there are ``Legendrian large" cables, that is cables with Thurston-Bennequin invariant larger than $pq$. Such cables will prevent the sort of analysis done here, and in addition, we know very little about solid tori with convex boundary having these dividing slopes. 

%----------------------------------------------------------
\subsection{Greater-sloped cables}
%----------------------------------------------------------
To state our results in this case, we need to recall the work of Chakraborty, the second and third authors  \cite{ChakrabortyEtnyreMin2024} on cabled knot types. Given a Legendrian knot $L$ in the knot type $K$ and relatively prime integers $p$ and $q$ with $q/p>\lceil w(K)\rceil$, we define the standard $(p,q)$-cable of $L$ as follows. Let $N$ be a standard neighborhood of $L$ with $\partial N$ having ruling slope $q/p$\footnote{We will assume the reader is familiar with standard results in contact geometry, such as standard neighborhoods of Legendrian knots and ruling curves, as can be found in \cite{ChakrabortyEtnyreMin2024, DaltonEtnyreTraynor2024}}, the \dfn{standard $(p,q)$-cable of $L$}, denoted $L_{(p,q)}$, is one of the ruling curves on $\partial N$. We define the \dfn{diamond of $L_{(p,q)}$} to be the collection of Legendrian knots
\[
  D(L_{(p,q)})=\{S_+^i\circ S_-^j(L_{(p,q)}): 0\leq i,j\leq p-1\}
\]
where $S_\pm(L)$ denotes the $\pm$-stabilization of $L$. 

It will be convenient to denote the set of Legendrian isotopy classes of Legendrian representatives of $K$ by $\mathcal{L}(K)$. 
The main results about greater-sloped cables\footnote{In \cite{ChakrabortyEtnyreMin2024}, it was called a \dfn{sufficiently positive cable}.} in \cite{ChakrabortyEtnyreMin2024} can be summarized as follows:

\begin{itemize}
  \item $\tb(L_{(p,q)})=pq-|p\tb(L)-q|$ and $\rot(L_{(p,q)})=q\rot(L)$,
  \item $\mathcal{L}(K_{(p,q)})= \cup_{L\in \mathcal{L}(K)} D(L_{(p,q)})$,
  \item if $L$ and $L'$ are not isotopic then $D(L_{(p,q)})$ and $D(L'_{(p,q)})$ are disjoint, and
  \item $S_\pm^p(L_{(p,q)})=(S_\pm(L))_{(p,q)}$.
\end{itemize}

It is not hard to see that this provides a complete classification Legendrian representatives of $K_{(p,q)}$ and their relation under stabilization, in terms of the same understanding of Legendrian representatives of $K$. From above, it is clear that for any element $L'\in \mathcal{L}(K_{(p,q)})$ there is a unique Legendrian $L\in \mathcal{L}(K)$ such that $L'\in D(L_{(p,q)})$. We call $L$ the \dfn{underlying knot} of $L'$. The classification says that two elements in $\mathcal{L}(K_{(p,q)})$ are isotopic if and only if they have the same underlying knot type, Thurston-Bennequin invariant, and rotation number. 

Given a Legendrian knot $L$ we define the \dfn{cone of $L$} to be 
\[
  C(L)=\{S^i_+\circ S^j_-(L): i,j\in \Z_{\geq0}\},
\]
that is $C(L)$ is the set of all Legendrian knots obtained from $L$ by stabilization. By the well-definedness of stabilization, it is clear that two Legendrian links in $C(L)$ are isotopic if and only if they have the same rotation number and Thurston-Bennequin invariant. 

Now we consider Legendrian links. Let $L$ be a Legendrian knot in the knot type $K$ and $N$ a standard neighborhood of $L$ with ruling slope $q/p$, where $q/p > \lceil \omega(K) \rceil$, as before. We define the \dfn{standard $(np,nq)$-cable link $L_{n(p,q)}$} to be a union of $n$ distinct ruling curves on $\partial N$. We can also define $C(\Lambda)$ for a Legendrian link $\Lambda$ as the set of all stabilizations of $\Lambda$ (including $\Lambda$ itself). An important property of the cone is its behavior under cabling.  

\begin{lemma}\label{cones}
  Given a knot type $K$, relatively prime integers $p$ and $q$ such that $q/p>\lceil w(K)\rceil$, and $n>1$, we have the following.
  \begin{enumerate}
    \item For $L\in\mathcal{L}(K),$ and $L'\in C(L)$, we have $C(L'_{n(p,q)})\subset C(L_{n(p,q)})$. 
    \item If $\Lambda \in C(L_{n(p,q)})$ then $\Lambda$ lies on $\partial N$ where $N$ is a standard neighborhood of $L$.
  \end{enumerate}
\end{lemma}

We can now state the classification of Legendrian knots of greater-sloped cables. 

\begin{theorem}\label{maingsc}
  Given a knot type $K$ let $L^1,\ldots, L^k$ denote the non-destabilizable Legendrian knots in $\mathcal{L}(K)$. For relatively prime integers $p$ and $q$ such that $q/p>\lceil w(K)\rceil$ and $n>1$, we have 
  \begin{enumerate}
    \item
    \[
      \mathcal{L}\left(K_{n(p,q)}\right)=\bigcup_{i=1}^k C\left((L^i)_{n(p,q)}\right).
    \]
  
    \item  Two Legendrian links in $C((L^i)_{n(p,q)})$ are isotopic if and only if they have the same classical invariants. 
    
    \item $\Lambda$ and $\Lambda'$ in $\mathcal{L}(K_{n(p,q)})$ are Legendrian isotopic if and only if there is some $L\in \mathcal{L}(K)$ such that $\Lambda$ and $\Lambda'$ are both in $C(L_{n(p,q)})$ and the components of $\Lambda$ and $\Lambda'$ have the same classical invariants.

    \item For $L,L' \in \mathcal{L}(K)$, $\Lambda\in C(L_{n(p,q)})\cap C(L'_{n(p,q)})$ if and only if there is some $L''\in C(L)\cap C(L')$ and $\Lambda \in C(L''_{n(p,q)})$. 
  \end{enumerate}
  
  Moreover, given any $\Lambda\in \mathcal{L}(K_{n(p,q)})$, any permutation of the components of $\Lambda$ preserving the classical invariants can be achieved by a Legendrian isotopy.  
\end{theorem}

We note that Item~(1) identifies all possible Legendrian links in $\mathcal{L}(K_{n(p,q)})$ while Items~(2) through~(4) determine with two Legendrian links from Item~(1) are Legendrian isotopic. Thus we have a complete classification. We will illustrate this with a simple example. 

\begin{example}
Consider the twist knot $K_{-5}$, see Figure~\ref{fig:twk}. We will describe the classification of Legendrian knots in terms of its mountain range. Recall we have the map
\[
\Phi: \mathcal{L}(K)\to \Z^2:L\mapsto (\rot(L),\tb(L)).
\]
The mountain range of $K$ is the image of $\Phi$ decorated with the number of elements in $\mathcal{L}(K)$ that map to a given point. Technically, we should also specify how each element in the mountain range behaves under stabilization, but in the examples we consider, this will be clear. From \cite{EtnyreNgVertesi13} we know the classification of Legendrian representatives of $K_{-5}$ is given on the left-hand side of Figure~\ref{fig:kn3}. 

\begin{figure}[htb]{
  \begin{overpic}%[grid,tics=10] 
  {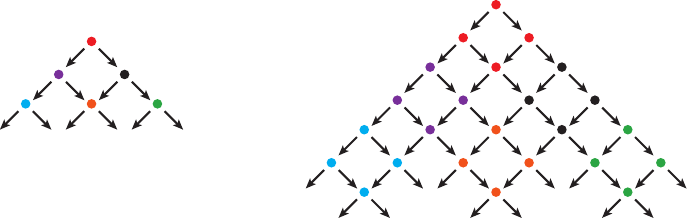}
  \put(244, 103){$2$}
  \put(260, 87){$2$}
  \put(213, 87){$2$}
  \put(236, 77){$2$}
  \put(50, 85){$2$}
  \end{overpic}}
  \caption{The mountain range for $K_{-5}$ on the left. There are two Legendrian representatives when $(\rot,\tb)=(0,-3)$ and one for all other points in the diagram. On the right is the mountain range for $(K_{-5})_{(2,1)}$. The peak of the mountain range is $(0,-5)$. The colors of the dots indicate the diamonds of the Legendrian representatives of $K_{-5}$.} 
  \label{fig:kn3}
\end{figure}
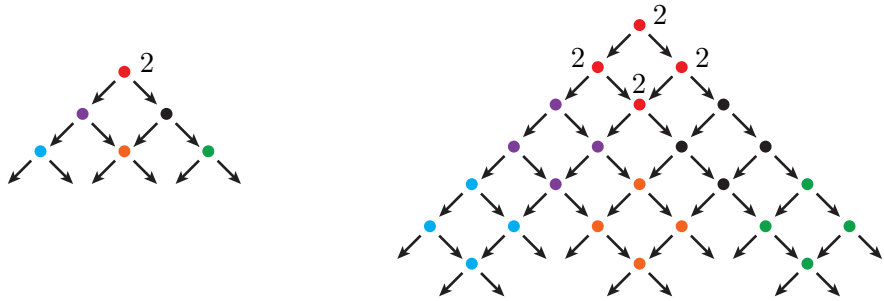

Specifically, there are two Legendrian representatives $L^1$ and $L^2$ with $(\rot,\tb)=(0,-3)$, and all other representatives are stabilizations of these and determined by their Thurston-Bennequin invariant and rotation number. 

Now consider the $(2,1)$-cable $(K_{-5})_{(2,1)}$. Its mountain range is shown on the right of the figure. According to our discussion of cables above, the top red dot represents the standard $(2,1)$-cables $L^1_{(2,1)}$ and $L^2_{(2,1)}$ while the other red dots are $S_+(L^1_{(2,1)}), S_-(L^1_{(2,1)}), S_+\circ S_-(L^1_{(2,1)})$ and $S_+(L^2_{(2,1)}), S_-(L^2_{(2,1)}), S_+\circ S_-(L^2_{(2,1)})$. All other dots are stabilizations of these and are determined by their classical invariants. 

We now consider the $(4,2)$-cable link. We have two representatives in $\mathcal{L}((K_{-5})_{2(2,1)})$ where all components have the maximal Thurston-Bennequin invariant. They are 
\[
  \Lambda^1 = \Lambda^1_1\cup \Lambda^1_2 = L^1_{2(2,1)} \;\text{ and }\; \Lambda^2 = \Lambda^2_1\cup \Lambda^2_2 = L^2_{2(2,1)}
\]
where each $\Lambda^1_i$ is isotopic to $L^1_{(2,1)}$ and each $\Lambda^2_i$ is isotopic to $L^2_{(2,1)}$. 

\begin{remark}
  We are already seeing a restriction on the components of Legendrian representatives of $(K_{-5})_{2(2,1)}$ coming from the classification. {\em A priori}, one might think there is a Legendrian realization of $(K_{-5})_{2(2,1)}$ with one component being isotopic to $L^1_{(2,1)}$ and the other isotopic to $L^2_{(2,1)}$, but this is not the case! While restrictions on the components of Legendrian realizations of links have been seen before \cite{DaltonEtnyreTraynor2024}, this is the first time we see restrictions on the components of the cable that go beyond the classical invariants. 
\end{remark}

According to (1) of Theorem \ref{maingsc}, all other elements in $\mathcal{L}((K_{-5})_{2(2,1)})$ are stabilizations of $\Lambda^1$ and $\Lambda^2$ which we denote by $\Lambda^i_{m,n,k,l}$, by which we mean the result of stabilizing the first component of $\Lambda^i$, $m$ times positively and $n$ times negatively, while the second component is stabilized $k$ times positively and $l$ times negatively. 

If $m,n \leq 1$ or $k,l \leq 1$, then $\Lambda^1_{m,n,k,l}$ and $\Lambda^2_{m,n,k,l}$ are different since their components are not Legendrian isotopic. If $m,n,k,l \geq 2$, then $\Lambda^1_{m,n,k,l}$ and $\Lambda^2_{m,n,k,l}$ are determined by their classical invariants since they are in the cone $C((S_+\circ S_- (L^1))_{2(2,1)}) = C((S_+\circ S_-(L^2))_{2(2,1)})$. We can also see that if $m,k \geq 2$, then $\Lambda^1_{m,n,k,l}$ and $\Lambda^2_{m,n,k,l}$ are determined by their classical invariants since they are both in the cone $C((S_+(L^1))_{2(2,1)})=C((S_+(L^2))_{2(2,1)})$. Of course, a similar result holds when $n,l\geq 2$.  

Lastly we consider the case $m,l\leq 1,$ $n,k\geq 2$ or $m,l\geq 2,$ $n,k\leq 1$. Then $\Lambda^1_{m,n,k,l}$ sits in $C(\Lambda^1)$ and $\Lambda^2_{m,n,k,l}$ sits in $C(\Lambda^2)$. Since $m,l \leq 1$ or $n,k \leq 1$, it is clear that both $\Lambda^1_{m,n,k,l}$ and $\Lambda^2_{m,n,k,l}$ do not sit in $C(L'_{2(2,1)})$ where $L'$ is either $S_{+}(L^1) = S_{+}(L^2)$ or $S_{-}(L^1) = S_{-}(L^2)$. That is, there is no $L \in C(L^1) \cap C(L^2)$ such that $\Lambda^i_{m,n,k,l} \in C(L_{2(2,1)})$ for $i=1,2$. By (4) of Theorem~\ref{maingsc}, $\Lambda^i_{m,n,k,l} \notin C(\Lambda^1) \cap C(\Lambda^2)$, and thus by (3) of Theorem~\ref{maingsc}, $\Lambda^1_{m,n,k,l}$ is not Legendrian isotopic to $\Lambda^2_{m,n,k,l}$. but it is clear that each component of one link is Legendrian isotopic to a component of the other link. This completes the classification of Legendrian representatives of $(K_{-5})_{2(2,1)}$ and establishes a new phenomenon for Legendrian links.  

\begin{proof}[Proof of Theorem~\ref{newpenom}]
  The last example above gives examples as claimed in the theorem.
\end{proof}
\end{example}

\begin{remark}
  We note that one can find many similar, but much more complex, examples by considering greater-sloped cables of connect sums of negative torus knots, cables of positive torus knots, and other twist knots. 
\end{remark}

We finish this section by noting that any permutation of the components of a greater-sloped cable can be achieved by a Legendrian isotopy if the classical invariants are preserved. 

\begin{theorem}[Ordered Classificaiton]\label{greaterordered}
  Let $K$ be a knot type and $p$ and $q$ be relatively prime integers such that $q/p>\lceil w(K)\rceil$. Given any element $\Lambda\in \mathcal{L}(K_{n(p,q)})$ any permutation of the components that respects the classical invariants of the components can be realized by a Legendrian isotopy.
\end{theorem}

\begin{proof}
  This theorem follows directly from the proof of Lemma~7.11 in \cite{DaltonEtnyreTraynor2024}. The lemma assumes the knot type is uniformly thick, but for cables with slope larger than $\omega(L)$ it also holds. 
\end{proof}

%----------------------------------------------------------
\subsection{Integer lesser-sloped and \texorpdfstring{$\tbb(K)$}{tbb(K)}-sloped cables}
%----------------------------------------------------------
We begin with some standard constructions of integer lesser-sloped cables. Given a Legendrian knot $L$, we can consider a standard neighborhood $N$ of $L$, and in the standard model, there is an annulus $A$ that contains $L$, and its characteristic foliation is linear, and each leaf is Legendrian isotopic to $L$. The \dfn{$n$-copy of $L$} is the link $nL$ consisting of $n$ leaves in the characteristic foliation of $A$. Alternatively, if $L$ is in $(\R^3, \ker(dz-y\, dx))$ then one can take the front projection of $L$ and $nL$ has front projection obtained by shifting $L$ slightly in the $z$-direction $n-1$ times. See Figure~\ref{fig:ncopy}. We note that $nL$ is an $(n, n\tb(L))$-cable of $L$. 

\begin{figure}[htb]{
\begin{overpic}%[grid,tics=10] 
{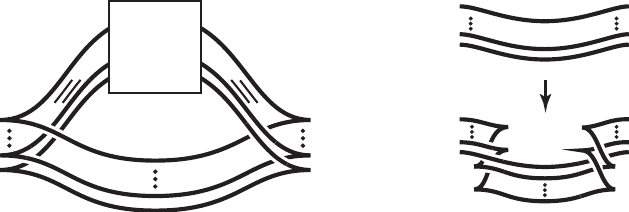}
\Large
\put(69, 73){$L$}
\end{overpic}}
\caption{The $n$-copy of a Legendrian knot $L$ on the left. To create the $t$-twisted $n$-copy of $L$, replace $t$ regions of the $n$-copy as shown at the top of the right-hand diagram with the region shown on the bottom.} 
\label{fig:ncopy}
\end{figure}

The \dfn{$t$-twisted $n$-copy of $L$}, denoted $T^t(nL)$, is the link consisting of $L$ and $(n-1)$ ruling curves of slope $\tb(L)-t$ on $\partial N$. Alternatively, if $L$ is in $(\R^3, \ker(dz-y\, dx))$, then one can take the front projection of the $n$-copy $nL$ of $L$ and then replace a region in the front projection as shown in Figure~\ref{fig:ncopy}. It is clear that $T^t(nL)$ is an $(n, n(t+\tb(L))$-cable of $L$. The fact that these two descriptions of the $t$-twisted $n$-copy of $L$ are the same was established in \cite[Remark~5.8]{DaltonEtnyreTraynor2024}.

When $\Lambda$ is an $n$-component Legendrian link and we have labeled the components $\Lambda_1\cup\cdots\cup \Lambda_n$, then we denote the result of $\pm$-stabilizing the $i^{th}$ component of $\Lambda$ by $S_{i,\pm}(\Lambda)$. Given the twisted $n$-copy $T^t(nL)$ of $L$, we say the component isotopic to $L$ by the first component and then label the others, cyclically, by $2$ through $n$. (When we say cyclically, we order them as they occur on $\partial N$.) 
\begin{lemma}\label{stabrelation}
We have the following relations between the stabilizations of the twisted $n$-copies:
\[
  S_{1,\pm}(T^t(nL)) = S_{2,\mp}\circ\cdots\circ S_{n,\mp} (T^{t-1}(nS_\pm(L))).
\]
In addition, we have 
\[
  S_{1,\pm}(T^1(nL))=S_{2,\mp}\circ\cdots\circ S_{n,\mp} (nS_\pm(L)).
\]
\end{lemma}

We are now ready to state the classification of Legendrian links in integer-sloped lesser-sloped cables of uniformly thick knots. 

\begin{theorem}\label{maininteger}
  Let $K$ be a uniform thick knot type and $q\leq \tbb(K)$ an integer. Let $L_1,\ldots, L_k$ be the distinct Legendrian representatives of $K$ with $\tb\geq q$. Then the non-destabilized Legendrian representatives of $K_{n(1,q)}$ are 
  \[
    \{ T^t(nL_i) : t=\tb(L_i)-q \text{ and } i=1, \ldots, k\}.
  \]
  Any other Legendrian representative of $K_{n(1,q)}$ destabilizes to one of these. 

  Let $\Lambda^1$ and $\Lambda^2$ be two Legendrian representatives of $K_{n(1,q)}$ that are not stabilizations of $n$-copies of a Legendrian $L_i$ with $\tb=q$. Let $\Lambda_1^i$ be the component of $\Lambda^i$ with the largest Thurston-Bennequin invariant (if there is more than one such component, then choose one), then $\Lambda^1$ and $\Lambda^2$ are Legendrian isotopic if and only if $\Lambda^1_1$ and $\Lambda^2_1$ are Legendrian isotopic and the other components of $\Lambda^1$ and $\Lambda^2$ can be paired to have the same Thurston-Bennequin invariant and rotation number. 

  Let $\Lambda^1$ and $\Lambda^2$ be two Legendrian representatives of $K_{n(1,q)}$ such that $\Lambda^1$ is obtained from a $n$-copy of $L_i$ with $\tb(L_i)=q$ by stabilizing at least one of the components positively some number of times (but not negatively) and at least one of the components negatively some number of times (but not positively) and similarly $\Lambda^2$ is obtained from an $n$-copy of $L_j$ with $\tb(L_j)=q$ by such stabilizations. If $i\not=j$, then $\Lambda^1$ and $\Lambda^2$ are not Legendrian isotopic. 

  Let $\Lambda^1$ and $\Lambda^2$ be two Legendrian representatives of $K_{n(1,q)}$ such that $\Lambda^1$ is obtained from a $n$-copy of $L_i$ with $\tb(L_i)=q$ by stabilizing all the components positively some number of times (but not negatively) and $\Lambda^2$ is obtained by the same stabilizations of an $n$-copy of $L_j$ with $\tb(L_j)=q$. Then $\Lambda^1$ and $\Lambda^2$ are Legendrian isotopic if and only if the components with the largest Thurston-Bennequin invariant are Legendrian isotopic. The same statement is true when all components are stabilized negatively. 
\end{theorem}

\begin{remark}
  We note that the last two results allow us to enumerate all distinct elements in $\mathcal{L}(K_{n(1,q)})$, except possibly those that are obtained from an $n$-copy by stabilizing all the components positively and negatively, but this can usually be understood by the proof of the above theorem. We see below that we can use this to obtain a complete classification of elements in $\mathcal{L}(K_{n(1,q)})$  for some knot types $K$. 
\end{remark}

To give an application of the previous theorem, we consider cable links of twist knots $K_{n}$. We first recall that the classification of Legendrian representatives of negative, even twist knots $\tbb(K_{-2n})=1$ is given in Figure~\ref{fig:secex}, see \cite{EtnyreNgVertesi13}. 

\begin{figure}[htb]{
\begin{overpic}%[grid,tics=10] 
{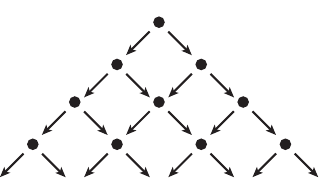}
\small
\put(83, 79){$\left\lceil \frac{n^2}{2}\right\rceil$}
\put(102, 58){$\lceil \frac n 2\rceil$}
\put(122, 40){$\lceil \frac n 2\rceil$}
\put(142, 19){$\lceil \frac n 2\rceil$}
\put(37, 58){$\lceil \frac n 2\rceil$}
\put(16, 40){$\lceil \frac n 2\rceil$}
\put(-6, 19){$\lceil \frac n 2\rceil$}
\put(74, 44){$1$}
\put(54, 23){$1$}
\put(95, 23){$1$}
\end{overpic}}
\caption{The mountain range for $K_{-2n}$. The peak of the mountain range is $(0,1)$.} 
\label{fig:secex}
\end{figure}
We denote the elements of $\mathcal{L}(K_{-2n})$ by 
\begin{align*}
L_{(r,t)} \quad &\text{ for $t\leq 0$ and $r=t+1, t+3, \ldots, -t-1$},\\
L_{(t-1,t)}^i, L_{(-t+1,t)}^i \quad &\text{ for $t\leq 0$ and $i=1,\ldots, k=\lceil n/2\rceil$, and}\\
L_{(0,1)}^i \quad &\text{for $i=1,\ldots, l=\lceil n^2/2\rceil$,}
\end{align*}
where $\tb(L_{(r,t)})=t, \rot(L_{(t,r)})=r$ and the $L_{(r,t)}$ are determined by their classical invariants, $\tb(L_{(\pm(t-1)),t}^i)=t,\rot(L_{(\pm(t-1),t)}^i)=\pm(t-1)$, and the $L_{(\pm(t-1),t)}^i$ are distinct for different $i$, and $\tb(L_{(0,1)}^i)=1, \rot(L_{(0,1)}^i)=0$, and the $L_{(0,1)}^i$ are distinct for different $i$. Moreover, $S_\mp(L_{(\pm(t-1),t)}^i)=L^i_{(\pm(t-2),t-1)}$, $S_+(L^i_{(t-1,t)})=L_{(t, t-1)}$, $S_-(L^i_{(1-t,t)})=L_{(-t,t-1)}$, and the set $\{S_+(L_{(0,1)}^i)\}_{i=1}^l$ agrees with the set $\{L_{(1,0)}^i\}_{i=1}^k$ and similarly for negative stabilizations of the $L_{(0,1)}^i$. 
 
\begin{theorem}\label{integerex}
  Consider the twist knot $K_{-2n}$ and any integer $q\leq 1$. Then, using the notation above, the non-destabilizable representatives of $\mathcal{L}((K_{-2n})_{m(1,q)})$ are
  \begin{align*}
    m L_{(r,q)}, &mL_{(\pm(q-1),q)}^i, T^1 (mL_{(r, q+1)}), T^1(m(L_{(\pm(q),q+1)}^i), \ldots,\\ 
    & T^{q-1}(mL_{0,-1}), T^{q-1}(mL_{(\pm 2, -1)}), T^{q}(mL^i_{(\pm 1, 0)}), T^{q+1}(mL_{(0,1)}^i)
  \end{align*}
  and all other representatives destabilize to one of these. 

  An element of $\mathcal{L}((K_{-2n})_{m(1,q)})$ is determined by the classical invariants of its components if the rotation number of each component has magnitude less than its Thurston-Bennequin invariant. If $q\not=1$ then all other representatives are determined by the Legendrian isotopy class of the component with the largest Thurston-Bennequin invariant. When $q=1$, then stabilizations of $m$-copies of distinct $L_{(0,1)}^i$ remain distinct until all components have been stabilized both positively and negatively, at which point they are determined by their classical invariants. 
\end{theorem}

\begin{proof}[Second proof of Theorem~\ref{newpenom}]
  When $n=4$ (or larger) one can see from \cite{EtnyreNgVertesi13} that there is an $i$ and $j$ such that $L^i_{(0,1)}$ and $L^j_{(0,1)}$ become isotopic after a single positive stabilization and also after a single negative stabilization. But according to the previous theorem $S_{1,+}\circ S_{2,-}(2L^i_{(0,1)})$ and $S_{1,+}\circ S_{2,-}(2L^j_{(0,1)})$ are not Legendrian isotopic even though they are component-wise isotopic. 
\end{proof}

We now turn to the possible permutations of the components of integer lesser-sloped cables that are realizable by a Legendrian isotopy. 

\begin{theorem}[Ordered Classification]\label{integerlesserslopeordered}
  Let $K$ be a uniform thick knot type and $q\leq \tbb(K)$ an integer. Given an element $\Lambda\in \mathcal{L}(K_{n(1,q)})$ then:
  \begin{enumerate}
    \item Any permutation of the components of $\Lambda$ are can be realized by a Legendrian isotopy unless $\Lambda= nL$ or $nL$ with some, but not all, components stabilized.
    \item If $\Lambda=nL$, possibly with some components stabilized, and $q=\tbb(K)$, then no permutations of the maximal Thurston-Bennequin invariant components can be realized by a Legendrian isotopy, but non-maximal Thurston-Bennequin invariant components can be permuted if they have the same classical invariants.
    \item If $\Lambda=nL$, possibly with some components stabilized, and $q<\tbb(K)$, then $\Lambda$ sits on a convex torus bounding a solid torus in the knot type of $K$ and the components of $\Lambda$ have a cyclic ordering given by this torus. Only cyclic permutations of the components of $\Lambda$ with maximal Thurston-Bennequin invariant can be realized by Legendrian isotopy, but non-maximal Thurston-Bennequin invariant components can be permuted if they have the same classical invariants.
  \end{enumerate}
\end{theorem}

%----------------------------------------------------------
\subsection{Non-integer lesser-sloped cables}
%----------------------------------------------------------
We begin this section with some standard constructions of lesser-sloped cables. Given a Legendrian knot $L$, we can consider a standard neighborhood $N$ of $L$. Inside of $N$ we can consider the stabilization $S_\pm(L)$ of $L$, and we let $N_\pm$ be a standard neighborhood of $S_\pm(L)$ in $N$. For any $q/p\in (\tb(L)-1,\tb(L))$ we can find a unique, up to contact isotopy, torus $T_\pm$ in $N\setminus N_\pm$ isotopic to $\partial N$ whose characteristic foliation is linear of slope $q/p$. We define the \dfn{$\pm$ standard $(np,nq)$-cable of $L$} to be $n$ leaves in the characteristic foliation of $T_\pm$ and we denote this link by $L^\pm_{n(p,q)}$. Alternatively, if $L$ is in $(\R^3, \ker(dz-y\, dx))$ then one can take the front projection of $L$ and then $L^\pm_{n(p,q)}$ has front projection given by the $np$ copy of $L$ after replacing a region as in the upper right of Figure~\ref{fig:ncopy} by $n(\tb(L)p-q)$ copies of the diagram on the left or the right of Figure~\ref{fig:negcable}.

\begin{figure}[htb]{
\begin{overpic}%[grid,tics=10] 
{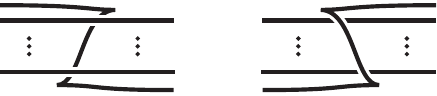}
\Large
\end{overpic}}
\caption{Legendrian tangles used to define standard Legendrian lesser-sloped cables of a Legendrian knot.} 
\label{fig:negcable}
\end{figure}

We begin by proving a theorem about lesser-sloped cable knots. The paper \cite{ChakrabortyEtnyreMin2024} describes such cables, but the theorem there requires quite a bit of knowledge about solid tori in the knot type that is not usually readily available. The theorem presented here will allow us to understand most cables of twist knots (among others).

\begin{theorem}\label{lesserknots}
  Let $K$ be a uniformly thick knot type and $p$ and $q$ relatively prime with $q/p<\tbb(K)$. Let $L_1, \ldots, L_k$ be the distinct Legendrian knots in $\mathcal{L}(K)$ with $\tb=\lceil q/p\rceil$. Every element in $\mathcal{L}(K_{(p,q)})$ destabilizes to $(L_i)^\pm_{(p,q)}$ for some $i$ and $\pm$.

  We distinguish the $(L_i)^\pm_{(p,q)}$ as follows.
  \begin{enumerate}
    \item If $\rot(L_i)\not= \rot(L_j)$, then $(L_i)^\pm_{(p,q)}$ and $(L_j)^\pm_{(p,q)}$ are distinct. 
    \item If $\rot(L_i)=\rot(L_j)$ but $S_\pm(L_i)\not=S_\pm(L_j)$, then $(L_i)^\pm_{(p,q)}$ and $(L_j)^\pm_{(p,q)}$ are distinct. 
    \item If $\rot(L_i)=\rot(L_j)$ but Legendrian surgery in $L_i$ and $L_j$ yield distinct contact manifolds, then $(L_i)^\pm_{(p,q)}$ and $(L_j)^\pm_{(p,q)}$ are distinct. 
    \item Otherwise, it is not clear if $(L_i)^\pm_{(p,q)}$ and $(L_j)^\pm_{(p,q)}$ are distinct. 
  \end{enumerate} 

  We  distinguish stabilizations of the $(L_i)^\pm_{(p,q)}$ as follows.
  \begin{enumerate}
    \item Stabilizations of $(L_i)^+_{(p,q)}$ and $(L_i)^-_{(p,q)}$ remain distinct until the first knot has been  $-$-stabilized $|p\tb(L)-q|$ or more times, the second knot has been $+$-stabilized $|p\tb(L)-q|$ or more times, and the classical invariants match. In this case they are Legendrian isotopic.
    \item If $S_+(L_i)=S_-(L_j)$, then $(L_i)^+_{(p,q)}$ and $(L_j)^-_{(p,q)}$ will have distinct rotation numbers until the first has been $+$-stabilized $|p(\tb(L)-1)-q|$ times and the second has been $-$-stabilized $|p(\tb(L)-1)-q|$ times. After each has been stabilized as indicated, they become Legendrian isotopic. 
    \item If $S_\pm(L_i)=S_\pm(L_j)$, but Legendrian surgery on $L_i$ and $L_j$ yield distinct contact manifolds, then stabilizations of $(L_i)^\pm_{(p,q)}$ and $(L_j)^\pm_{(p,q)}$ remain distinct until each has been $\pm$-stabilized $|p(\tb(L)-1)-q|$ or more times, and the classical invariants are the same. In this case they are Legendrian isotopic.
    \item If $S^l_\pm(L_i)$ is not Legendrian isotopic to $S^l_\pm(L_j)$ for any $l$, then any number of $\pm$-stabilizations of $(L_i)^\pm_{(p,q)}$ and $(L_j)^\pm_{(p,q)}$ remain distinct.
  \end{enumerate} 
\end{theorem}

We note that the statements in the theorem do not necessarily classify all Legendrian cable knots of uniformly thick knot types, but the theorem and its proof are sufficient in many cases. In particular, we have the following result. 

\begin{theorem}\label{knotcable}
  Let $K$ be the $-2n$-twist knot, $-q/p\in(-m-1,-m)$ for $m\geq0$ an integer, and $k=\lceil\frac n2\rceil$. Any $(r,t)$ in the mountain range of $\mathcal{L}(K_{(p,-q)})$ with $t\leq pq$ and $r\in [p-q+(pq+t), -q+p+(t-pq)]$ is realized by a unique Legendrian knot and any other pair $(r',t')$ in the mountain range is realized by $k$ distinct Legendrian knots. 

  More specifically, there are $2m+4k$ Legendrian knots in $\mathcal{L}(K_{(p,q)})$ with $\tb=pq$. Of those, $2m$ are determined by their rotation numbers, which are
  \[
    \pm (p-q+2pl)\quad \text{ for } l=0,\ldots m-1,
  \] 
  and for each rotation number 
  \[
    \rot = \pm (p+q) \text{ and } \pm ((2m+1)p-q).
  \]
  there are $k$ Legendrian knots. All other Legendrian knots are stabilizations of these. The mountain range for $K_{(p,q)}$ is given in Figure~\ref{fig:cablemr}. 
\end{theorem}

\begin{figure}[htb]{
\begin{overpic}%[grid,tics=10] 
{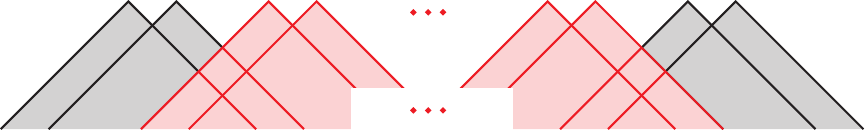}
\end{overpic}}
\caption{The mountain range for $K_{(p,q)}$ where $K$ is a $-2n$-twist knots and $q/p\in(-m-1,-m)$. Any point $(r,t)$ with $r+t$ odd in the red region is realized by a unique Legendrian knot, while in the black region is realized by $k$ Legendrian knots, where $k$ is as in Theorem~\ref{knotcable}. The peaks are at $\tb=pq$ and the values of the rotation numbers at the peaks are given in Theorem~\ref{knotcable}.} 
\label{fig:cablemr}
\end{figure}

The only range of cabling slopes in which we cannot classify cables for $-2n$ twist knots is $(0,1)$, but under one assumption, we can understand this range too. 

\begin{theorem}\label{ifsurg}
  Let $K$ be the $-2n$-twist knot, $q/p\in(0,1)$, $k=\lceil\frac n2\rceil$ and $l=\lceil\frac{n^2}{2}\rceil$. Suppose that Legendrian surgery on the maximal Thurston-Bennequin invariant Legendrian representatives of $K$ yields distinct contact manifolds. Then any $(r,t)$ in the mountain range of $\mathcal{L}(K_{(p,q)})$ with $r\in [-(pq-p-q-t), pq-p-q-t]$ is realized by a unique Legendrian knot, any such pair with $r\in [\pm p -(pq-q-t), \pm p+(pq-q-t)]$ that is not in the first region has exactly $k$ Legendrian representatives, and any other pair of $(r,t)$ in the mountain range is realized by $l$ distinct Legendrian knots. 

  More specifically, there are $2l$ distinct Legendrian knots $L^\pm_i, i=1,\ldots, l$ in $\mathcal{L}(K_{(p,q)})$ with $\tb=pq$ and $\rot(L^\pm_i)=\pm (p-q)$. All other Legendrian knots in $\mathcal{L}(K_{(p,q)})$ destabilize to one of these. In addition, $S^{p-q}_-(L^+_i)=S^{p-q}_+(L^-_i)$, the collection $\{L^+_i\}_{i=1}^l$ after $+$-stabilizing $q$ times has $k$ distinct members and similarly for $\{L^-_i\}_{i=1}^l$ after $-$-stabilizing $q$ times. After the $L_i^\pm$ have been stabilized into the first regions mentioned above, then they are determined by their classical invariants. The mountain range $K_{(p,q)}$ is given in Figure~\ref{fig:cablemr2}. 
\end{theorem}

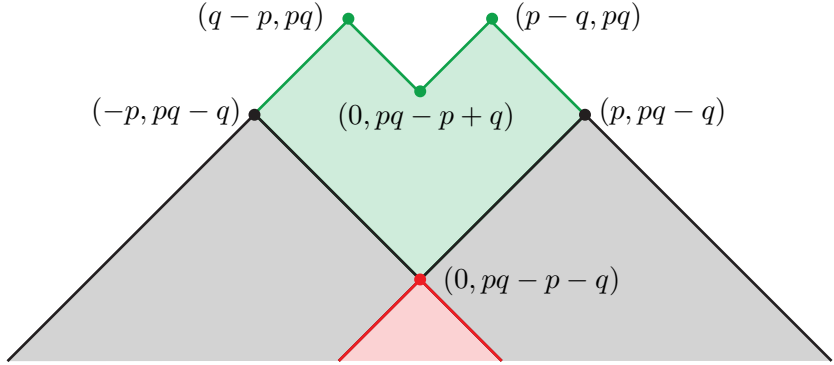
\begin{figure}[htb]{
\begin{overpic}%[grid,tics=10] 
{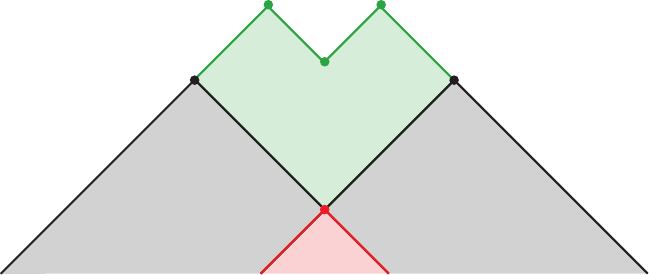}
\put(192, 128){$(p-q,pq)$}
\put(72, 128){$(q-p,pq)$}
\put(125, 90){$(0,pq-p+q)$}
\put(224, 92){$(p,pq-q)$}
\put(32, 92){$(-p,pq-q)$}
\put(165, 28){$(0,pq-p-q)$}
\end{overpic}}
\caption{The $(p,q)$-cables of $K_{-2n}$ where the $(r,t)$ with $r+t$ odd in the upper region have $l=\lceil\frac{n^2}2\rceil$ Legendrian representatives, such pairs in the regions on the right and left have $k=\lceil \frac n2\rceil$ Legendrian representatives, and such pairs in the bottom region have a unique Legendrian representative.} 
\label{fig:cablemr2}
\end{figure}

We note that in \cite{BourgeoisEkholmEliashberg12}, the hypothesis of the previous theorem was established for $n=2$. While it is likely that the techniques in that paper established the hypothesis for all $n>2$ too, we have not been able to verify that. 

We now turn to lesser-sloped, non-integral cable links. 

\begin{theorem}\label{negative}
  Let $K$ be a uniformly thick knot type and $p$ and $q$ relatively prime with $q/p<\tbb(K)$. Let $L_1, \ldots, L_k$ be the distinct Legendrian knots in $\mathcal{L}(K)$ with $\tb=\lceil q/p\rceil$. Every element in $\mathcal{L}(K_{n(p,q)})$ destabilizes to $(L_i)^\pm_{n(p,q)}$ for some $i$ and $\pm$.

  We distinguish the $(L_i)^\pm_{n(p,q)}$ as follows.
  \begin{enumerate}
    \item If $\rot(L_i)\not= \rot(L_j)$, then $(L_i)^\pm_{n(p,q)}$ and $(L_j)^\pm_{n(p,q)}$ are distinct. 
    \item If $\rot(L_i)=\rot(L_j)$ but $S_\pm(L_i)\not=S_\pm(L_j)$, then $(L_i)^\pm_{n(p,q)}$ and $(L_j)^\pm_{n(p,q)}$ are distinct. 
    \item If $\rot(L_i)=\rot(L_j)$ but Legendrian surgery in $L_i$ and $L_j$ yield distinct contact manifolds, then $(L_i)^\pm_{n(p,q)}$ and $(L_j)^\pm_{n(p,q)}$ are distinct. 
    \item Otherwise, it is not clear if $(L_i)^\pm_{n(p,q)}$ and $(L_j)^\pm_{n(p,q)}$ are distinct. 
  \end{enumerate} 

  We  distinguish stabilizations of the $(L_i)^\pm_{n(p,q)}$ as follows.
  \begin{enumerate}
    \item Stabilizations of $(L_i)^+_{n(p,q)}$ and $(L_i)^-_{n(p,q)}$ remain distinct until each component of the first link has been  $-$-stabilized $|p\tb(L)-q|$ or more times, each component of the second link has been $+$ -stabilized $|p\tb(L)-q|$ or more times, and the classical invariants of the components of one link match those of the other, in this case they are Legendrian isotopic.
    \item If $S_+(L_i)=S_-(L_j)$, then $(L_i)^+_{n(p,q)}$ and $(L_j)^-_{n(p,q)}$ will have distinct rotation numbers until each component of the first has been $+$-stabilized $|p(\tb(L)-1)-q|$ times and each component of the second has been $-$-stabilized $|p(\tb(L)-1)-q|$ times. After each has been stabilized as indicated, they become Legendrian isotopic. 
    \item If $S_\pm(L_i)=S_\pm(L_j)$, but Legendrian surgery on $L_i$ and $L_j$ yield distinct contact manifolds, then stabilizations of $(L_i)^\pm_{n(p,q)}$ and $(L_j)^\pm_{n(p,q)}$ remain distinct until each component of the of both links have been $\pm$-stabilized $|p(\tb(L)-1)-q|$ or more times, and the classical invariants of the components of one link match those of the other, in this case they are Legendrian isotopic.
    \item If $S^l_\pm(L_i)$ is not Legendrian isotopic to $S^l_\pm(L_j)$ for any $l$, then any number of $\pm$-stabilizations of $(L_i)^\pm_{n(p,q)}$ and $(L_j)^\pm_{n(p,q)}$ remain distinct.
  \end{enumerate} 
\end{theorem}

The statements in the theorem do not necessarily classify all Legendrian cable knots of uniformly thick knot types, but the theorem and its proof are sufficient in many cases. In particular, we have the following result. 

\begin{theorem}\label{lesserlink}
  Let $K$ be the $-2n$-twist knot, $p$ and $q$ relatively prime with $-q/p<1$ and $n>1$ an integer. Let $L_1, \ldots, L_k$ be the distinct Legendrian knots in $\mathcal{L}(K)$ with $\tb=\lceil -q/p\rceil$. Every element in $\mathcal{L}(K_{n(p,-q)})$ destabilizes to $(L_i)^\pm_{n(p,-q)}$ for some $i$ and $\pm$. Two Legendrian links in $\mathcal{L}(K_{n(p,-q)})$ are Legendrian isotopic if and only if they are component-wise isotopic. The classification of Legendrian realizations of $K_{n(p,q)}$ now follows from Theorem~\ref{knotcable} when $q/p<0$ and from Theorem~\ref{ifsurg} if Legendrian surgery on the maximal Thurston-Bennequin invariant representatives of $K$ are distinct. 
\end{theorem}

\subsection*{Acknowledgements.}
The authors would like to thank Hiro Lee Tanaka for pointing out the reference \cite{JordanTraynor06} and the existence of distinct Legendrian links that were smoothly isotopic and component-wise Legendrian isotopic. 
The first author is partially supported by SFB/TRR 191 ``Symplectic Structures in Geometry, Algebra and Dynamics, funded by the DFG (Project- ID 281071066-TRR 191)'' ,the Georgia Institute of Technology's Elaine M. Hubbard Distinguished Faculty Award, and the AWM Mentoring Grant. This project started when the first author was visiting Georgia Tech. She thanks them for their hospitality.  The second author was partially supported by National Science Foundation grant DMS-2203312 and the Georgia Institute of Technology's Elaine M. Hubbard Distinguished Faculty Award. The fourth author was partially supported by National Science Foundation grant DMS-2203312.

%%%%%%%%%%%%%%%%%%%%%%%%%%%%%%%%%%%%
\section{Greater-sloped cables}\label{sec:greater}
%%%%%%%%%%%%%%%%%%%%%%%%%%%%%%%%%%%%
In this section we prove all the stated results about greater-sloped cables. We start with Lemma~\ref{cones}, which says that $L\in\mathcal{L}(K)$, $q/p\geq \lceil \omega(L)\rceil$, and $L'$ is in the cone $C(L)$ of $L$, then we have $C(L'_{n(p,q)})\subset C(L_{n(p,q)})$. 

\begin{proof}[Proof of Lemma~\ref{cones}]
Suppose that $L'=S_+^k\circ S_-^l(L)$. We claim that $L'_{n(p,q)}$ is obtained from $L_{n(p,q)}$ by stabilizing each component $pk$ times positively and $pl$ times negatively. Given this, the claim that $C(L'_{n(p,q)})\subset C(L_{n(p,q)})$ follows. 

To see that the claim is true, we note that if $N$ is a standard neighborhood of $L$ then we can take the standard neighborhood $N'$ of $L'$ to be inside of $N$. Recall that $L_{n(p,q)}$ sits on $\partial N$ as ruling curves and $L'_{n(p,q)}$ sits on $\partial N'$ as ruling curves. Let $A_1\cup \cdots \cup A_n$ be a collection of $n$ disjoint annuli such that each $A_i$ has one boundary component on a component of $L_{n(p,q)}$ and the other on a component of $L'_{n(p,q)}$. We can assume that the $A_i$ are convex and the dividing set of $A_i$ intersects $L_{n(p,q)}$ exactly $|p\tb(L)-q|$ times and $L'_{n(p,q)}$ exactly $|p\tb(L')-q|=|p(\tb(L)-k-l)-q|$ times. Thus, there are $p(k+l)$ boundary parallel dividing curves on $A_i$ giving bypasses of a component of $L'_{n(p,q)}$. Considering the rotation numbers of the components of $L'_{n(p,q)}$ and $L_{n(p,q)}$, we see there are $pk$ positive bypasses and $pl$ negative bypasses. Thus $L'_{n(p,q)}$ destabilizes to $L_{n(p,q)}$ as claimed. 

The second statement directly follows from the fact that the standard cable $L_{n(p,q)}$ lies on $\partial N$. Then any stabilization of $L_{n(p,q)}$ can be obtained by perturbing $L_{n(p,q)}$ to increase the intersection number with the dividing curves, so it will stay on $\partial N$. 
\end{proof}

We now turn to the main theorem concerning greater-sloped cables. To do so, we need the following definitions and lemma. Given a Legendrian link $\Lambda\in\mathcal{L}(K_{n(p,q)})$, we say that $L \in \mathcal{L}(K)$ is an \dfn{underlying Legendrian} of $\Lambda$ if $\Lambda$ is Legendrian isotopic to an $n$ copy of  $(p,q)$ curves on the boundary of a neighborhood of $L$. 

For a Legendrian link $\Lambda$ in $\mathcal{L}(K_{n(p,q)})$, we define a \dfn{minimal underlying Legendrian} $L_{min}$ to be an underlying Legendrian of $\Lambda$ such that neither $S_+(L_{min})$ nor $S_-(L_{min})$ is an underlying Legendrian of $\Lambda$. 

\begin{lemma}\label{lem:minimal}
Let $\Lambda$ be a Legendrian link in $\mathcal{L}(K_{n(p,q)})$. Then there exists a minimal underlying Legendrian knot $L_{min} \in \mathcal{L}(K)$ and if $L'_{min}$ is another minimal underlying Legendrian of $\Lambda$, then $\tb(L_{min}) = \tb(L'_{min})$ and $\rot(L_{min}) = \rot(L'_{min})$.
\end{lemma}

\begin{proof}
  The existence of $L_{min}$ is obvious. By the definition, we know $S_{\pm}(L_{min})$ are not underlying Legendrian of $\Lambda$. This implies at least one of the following two possibilities:
  \begin{enumerate}
    \item at least one component of $\Lambda$ is $(L_{min})_{(p,q)}$, or
    \item at least two components of $\Lambda$ are $S^i_+((L_{min})_{(p,q)})$ and $S^j_-((L_{min})_{(p,q)})$ for $i, j > 0$, respectively.
  \end{enumerate} 
  This means that one component of $\Lambda$ is at the top of the cone of $(L_{min})_{(p,q)}$ or there are two conponents of $\Lambda$, one of which is on the left boundary edge of the cone $(L_{min})_{(p,q)}$ and the other is on the right boundary edge.
  
  Now, assume there exists another minimal underlying knot $L'_{min}$ for $\Lambda$. Suppose that $\tb(L'_{min}) > \tb(L_{min})$. Also, without loss of generality, we may assume $\rot(L_{min}) \leq \rot (L'_{min})$. In this case, the conditions (1) above does not hold for $L'_{min}$ since otherwise there would exist a component of $\Lambda$ with $\tb = \tb((L'_{min})_{(p,q)}) > \tb((L_{min})_{(p,q)})$, which contradicts that $L_{min}$ is an underlying Legendrian of $\Lambda$. Thus only the condition (2) holds for $L'_{min}$, which implies that there exists a component of $\Lambda$ that lies on the right boundary edge of the cone $C((L'_{min})_{(p,q)})$ However, since $\tb(L_{min})< \tb(L'_{min})$ and $\rot(L_{min}) \leq \rot(L'_{min})$, the cone $C((L_{min})_{(p,q)})$ does not contain the right boundary edge of the cone $C((L'_{min})_{(p,q)})$, which again contradicts that $L_{min}$ is an underying Legendrian of $\Lambda$. 
 
  Now consider the case that $\tb(L_{min}) = \tb(L'_{min})$. Again, without loss of generality, we may assume $\rot(L_{min}) < \rot(L'_{min})$. In this case, $C((L_{min})_{(p,q)})$ cannot contain the right boundary edge of $C((L'_{min})_{(p,q)})$. That means neither Conditions~(1) or~(2) can hold, which is a contradiction.
\end{proof}

\begin{proof}[Proof of Theorem~\ref{maingsc}]
The last statement about isotoping the components of the cable link follows from \cite[Lemma~7.11]{DaltonEtnyreTraynor2024} and the fact that stabilization is a well-defined operation. (The lemma in \cite{DaltonEtnyreTraynor2024} assumes uniform thickness, but it is not used in the proof. That was an oversight in the paper due to the fact that many of the results in the paper needed that assumption.)

Recall we are taking $L^1,\ldots, L^k$ to be the non-destabilizable Legendrian knots in $\mathcal{L}(K)$ and relatively prime integers $p$ and $q$ such that $q/p>\lceil w(K)\rceil$.  

Item~(1) follows from Proposition~7.7 in \cite{DaltonEtnyreTraynor2024}. That proposition assumes that $K$ is uniformly thick, but that is only needed for lesser-sloped cables; the proof for greater-sloped cables does not use uniform thickness. 

Item~(2) follows as they are all stabilizations of a fixed Legendrian link and we know any permutations of the components of the link that respect the classical invariants can be realized by a Legendrian isotopy. 

Now we consider Item~(3). The reverse implication follows from Item~(2). We now assume that $\Lambda$ and $\Lambda'$ are Legendrian isotopic. Clearly they have the same classical invariants. From the first part of the theorem, there exists some $i$ such that $\Lambda \in C((L^i)_{n(p,q)})$. Since $\Lambda'$ is Legendrian isotopic to $\Lambda$, we know $\Lambda'$ is also in $C((L^i)_{n(p,q)})$.  

Finally, we consider Item~(4). The reverse implication follows from Item~(1) of Lemma~\ref{cones}. Now suppose that $\Lambda\in C(L_{n(p,q)})\cap C(L'_{n(p,q)})$. According to Lemma~\ref{lem:minimal}, there exist minimal underlying Legendrian knots $L_{min} \in C(L)$ and $L'_{min} \in C(L')$ such that $\tb(L_{min}) = \tb(L'_{min})$ and $\rot(L_{min})= \rot(L'_{min})$. Then by Item~(2) of Lemma~\ref{cones}, $\Lambda$ lies on $T$, the boundary of a standard neighborhood of $L_{min}$ and also lies on $T'$, the boundary of a standard neighborhood of $L'_{min}$. That is, the intersection of two tori $T$ and $T'$ contains $\Lambda$. There is a smooth isotopy between $T$ and $T'$ fixing $\Lambda$ (see, for example, \cite[Lemma~2.16]{ChakrabortyEtnyreMin2024}). Then by discretization of isotopy \cite{Colin97,Honda02}, there exists a sequence of convex tori $T = T_1, \ldots, T_n = T'$ such that $T_{i+1}$ can be obtained by a bypass attachment to $T_i$ where the bypass does not intersect $\Lambda$. Since $L_{min}$ is minimal, it is clear that the dividing slopes of $T_i$ are in $(\tb(L_{min})-1, \infty)$. We claim that the dividing slopes of $T_i$ are actually in $[\tb(L_{min}), \infty)$ by showing that they cannot be in $(\tb(L_{min})-1, \tb(L_{min}))$. To see this, we note that if $\Lambda$ sits on a torus $T''$ with dividing slope in $(\tb(L_{min})-1, \tb(L_{min}))$, then $T''$ bounds a solid torus which contains another torus $T'''$ with dividing slope $\tb(L_{min})-1$ and the ruling curves of slope $q/p$ on $T''$ are stablizations of those on $T'''$ (as can easily be seen by taking annuli between the ruling curves). Thus $\Lambda$ also sits on $T'''$, but this contradicts Lemma~\ref{lem:minimal} since $T'''$ is the boundary of a standard neighborhood of a Legendrian knot with Thurston-Bennequin invariant different from $L_{min}$. 

Now we inductively claim that each $T_i$ bounds a solid torus that contains (or is) a standard neighborhood of $L_{min}$. For $i=1$ the claim is true. Now, assume the claim is true for $i$. Now we can obtain $T_{i+1}$ by attaching a bypass to $T_i$. There are three possible non-trivial bypasses that can be attached to $T_i$: one changes the number of dividing curves, or one increases the slope of the dividing curves, and one decreases the slope. 

In the first case, the bypass can be attached from the outside or the inside (here, outside means that the bypass is outside the solid torus that $T_i$ bounds). If the bypass is attached from the outside, then $T_{i+1}$ bounds a solid torus that contains the solid torus that $T_i$ bounds and hence also contains a neighborhood of $L_{min}$. If the bypass were attached from the inside, then $T_{i+1}$ bounds a solid torus that contains a torus $T'_{i+1}$ isotopic to it such that $T_i$ and $T'_{i+1}$ cobound $T^2\times I$ with an $I$ invariant contact structure. Thus $T'_{i+1}$ bounds a solid torus containing a neighborhood of $L_{min}$ since $T_i$ does and $T_{i+1}$ bounds a solid torus containing the solid torus $T'_{i+1}$ bounds, completing the induction in this case. 

If the bypass increases the dividing slope, then it must be attached from the outside of $T_i$, so $T_{i+1}$ bounds a solid torus containing $T_{i}$, and we are done as above. 

Lastly, consider the case where the bypass decreases the dividing slope. That means the bypass is attached from inside of $T_i$, so the solid torus $S_i$ that $T_i$ bounds contains $T_{i+1}$. Let $s_i$ be the slope of the dividing curves on $T_i$. We consider two cases. First, suppose that $s_i$ is not an integer. According to \cite[Proposition 3.3]{Min24}, there is a unique solid torus $S'_i$ in $S_i$ up to contact isotopy with convex boundary having two dividing curves of slope $\lfloor s_i\rfloor$. By the induction hypothesis, there is a standard neighborhood $N_i$ of $L_{min}$ in $S_i$, and since $\tb(L_{min})\leq \lfloor s_i\rfloor$, we know that $N_i$ is inside of $S'_i$. Now since $T_{i+1}$ is obtained from $\partial S_i$ by a bypass attachment, its slope must be greater than or equal to $\lfloor s_i\rfloor$ and hence $S_{i+1}$ will also contain $N_i$, so we can let $N_{i+1}=N_i$. 

We are left to consider the case when $s_i$ is an integer. If $s_i=\tb(L_{min})$, then the slope of $T_{i+1}$ is also $s_i$ and the bypass attachment was trivial and did not decrease the dividing slope. So we can assume $s_i>\tb(L_{min})$. The solid torus $S_i$ is a standard neighborhood of a Legendrian knot $L''$, and since $S_i$ contains a standard neighborhood of $L_{min}$, we know that $L_{min}$ is a stabilization of $L''$. If we have to stabilize $L''$ both positively and negatively to obtain $L_{min}$, then one may easily check that no matter what $T_{i+1}$ is, the solid torus $S_{i+1}$ that it bounds will contain a neighborhood of $L_{min}$ (this is because it will contain a neighborhood of $S_+(L'')$ or $S_-(L)$ which both contain a neighborhood of $L_{min}$ in this case). However, if $L_{min}$ is obtained by stabilizing $L''$ with only one sign, say positively, then it is possible that attaching a bypass from the inside of $S_i$ will result in a torus that does not bound a solid torus containing a neighborhood of $L_{min}$. But if this is the case then $S_{i+1}$ does contain a neighborhood of $S_-(L'')$ and as argued above (in the first paragraph when discussing Item~(4)) we see that $\Lambda$ will sit on the boundary of the standard neighborhood of $S_-(L'')$. Thus $\Lambda$ is in $C((S_-(L''))_{n(p,q)})$. However, if one considers how $C((S_-(L''))_{n(p,q)})$ and $C((L_{min})_{n(p,q)})$ intersect, we see that $\Lambda$ cannot staisfy Conditions~(1) and~(2) in the proof of Lemma~\ref{lem:minimal} and thus $L_{min}$ is not a minimal underlying Legendrian knot for $\Lambda$. This contradiction implies that the $S_{i+1}$ does contain a neighborhood of $L_{min}$ as desired. 

We now notice that a standard neighborhood of $L'_{min}$ (bounded by $T'$) contains a standard neigbhorhood of $L_{min}$ (by our inductive hypothesis), so $L'_{min}$ must be isotopic to $L_{min}$. Thus $L_{min}$ is in the cone of $L$ and $L'$ finishing the proof. 
\end{proof}

%%%%%%%%%%%%%%%%%%%%%%%%%%%%%%%%%%%%
\section{Integer lesser-sloped cables}
%%%%%%%%%%%%%%%%%%%%%%%%%%%%%%%%%%%%
In this section we prove all the stated results about integer lesser-sloped cables. We start with Lemma~\ref{stabrelation}, which says that  the following relations between the stabilizations of the twisted $n$-copies hold
\[
S_{1,\pm}(T^t(nL)) = S_{2,\mp}\circ\cdots\circ S_{n,\mp} (T^{t-1}(nS_\pm(L)))
\]
and
\[
S_{1,\pm}(T^1(nL))=S_{2,\mp}\circ\cdots\circ S_{n,\mp} (nS_\pm(L)).
\]
\begin{proof}[Proof of Lemma~\ref{stabrelation}]
Recall that we form $T^t(nL)$ by taking $L$ and the union of $(n-1)$ ruling curves $L_2,\ldots, L_{n}$ of slope $\tb(L)-t$ on the boundary of a standard neighborhood $N(L)$ of $L$. Now let $N(S_\pm(L))$ be a standard neighborhood of $S_\pm(L)$ inside of $N(L)$. Note $S_\pm(L)$ together with $(n-1)$-curves $L'_2,\ldots, L'_n$ of slope $\tb(L)-t$, which is also equal to the slope $\tb(S_\pm(L))-(t-1)$, on $\partial N(S_\pm(L))$ is the $(t-1)$ twisted $n$-copy of $S_\pm(L)$. Moreover, we can take annuli between $L_i$ and $L'_i$ to see a bypass for $L_i'$. Thus the $L'_i$ are stabilizations of the $L_i$. As we know that rotation numbers of the $L_i'$ match that of $S_\pm(L)$ and those of $L_i$ match that of $L$, we see the signs of the stabilization are as claimed. 

The second relation between the $1$ twisted $n$-copy and the $n$-copy of a stabilization follows similarly, except to see the destabilizations of the we need to use \cite[Lemma~7.8]{EtnyreMinMukherjee22Pre}.
\end{proof}

Before proving the classification of integer lesser-sloped cables, we prove a useful lemma.
\begin{lemma}\label{forcedstabs}
Let $N$ be a standard neighborhood of a Legendrian knot $L$ and $L'$ a Legendrian divide on $\partial N$. Let $T$ be any convex torus inside of $N$ that is smoothly isotopic to $\partial N$, but has a dividing slope in $[\tb(L)-1, \tb(L))$. Then there is a sign $s$ such that any Legendrian curve on $T$ that is smoothly isotopic to $L'$ is Legendrian isotopic to a stabilization of $L'$, at least one of which is an $s$-stabilization. 

The same is true for any convex torus $T$ outside of $N$ that is smoothly isotopic to $\partial N$ but has dividing slope in $(\tb(L),\tb(L)+1]$.
\end{lemma}
\begin{proof}
We will choose a framing on $L$ so that the dividing curves on $\partial N$ have slope $0$. Given $T$ as in the first paragraph, we know that there is a solid torus $N'$ inside of the torus that $T$ bounds that has dividing slope $-1$  and is a standard neighborhood of an $s$-stabilization of $L$. Suppose it is a $s$-stabilization. From Lemma~2.2 in \cite{EtnyreHonda05} we know that a dividing curve of slope $q/p$ on $\partial T$ has rotation number
\[
q\rot(\mu)+p\rot(\lambda)
\]
where $\mu$ is a Legendrian realization of the meridian on $T$ and $\lambda$ is a Legendrian realization of a curve isotopic to $L$ on $T$ (here we are using this longitude to measure slopes on $T$). %Because the region $R$ between $\partial N$ and $\partial N'$ is a basic slice we know it has relative Euler class $s \begin{bmatrix} 1 & 0 \end{bmatrix}^T$. 
Now by the proof of Proposition~4.22 in \cite{Honda00a} the region $R$ between $\partial N$ and $T$ will have relative Euler class $s\begin{bmatrix} p-1 & q\end{bmatrix}^T$ (all the signs of basic slices making up $R$ must be the same since $R$ is a subset of the region between $\partial N$ and $\partial N'$ which is a basic slice). Thus, the difference between the rotation numbers of $L$ and a ruling curve of slope $0$ on $T$ is $sq$ and hence the ruling curve will have rotation number $\rot(L)+s q$ (we note that $q>0$ since the slope of the dividing curves is in $[\tb(L)-1, \tb(L))$). Moreover, again by \cite[Lemma~7.8]{EtnyreMinMukherjee22Pre}, we know that this ruling curve must be a $q$-times $s$-stabilization of $L'$. Any other curves on $T$ in this smooth isotopic class is a stabilization of this ruling curve. 
\end{proof}

We now turn to the classification of integer lesser-sloped cables.
\begin{proof}[Proof of Theorem~\ref{maininteger}]
The claim that any Legendrian representative of $K_{n(1,q)}$ destabilizes to a $t$-twisted $n$-copy of a non-destabilizable Legendrian representative of $K$ follows directly from Proposition~7.8 in \cite{DaltonEtnyreTraynor2024}.

Let $\Lambda^1$ and $\Lambda^2$ be two Legendrian representatives of $K_{n(1,q)}$ that are not $n$-copies of a Legendrian $L_i$ with $\tb=q$. Let $\Lambda_1^i$ be the component of $\Lambda^i$ with the largest Thurston-Bennequin invariant (if there is more than one such component, then choose one). We will now show that $\Lambda^1$ and $\Lambda^2$ are Legendrian isotopic if and only if $\Lambda^1_1$ and $\Lambda^2_1$ are Legendrian isotopic and the other components of $\Lambda^1$ and $\Lambda^2$ can be paired to have the same Thurston-Bennequin invariant and rotation number. The forward implication is clear (if there are more than one component of the link with largest Thurston-Bennequin invariant then we need to know that these components can be permuted by a Legendrian isotopy, but this follows from the ordered classification in Theorem~\ref{integerlesserslopeordered}). 

For the reverse implication, we assume that $\Lambda^1_1$ and $\Lambda^2_1$ are Legendrian isotopic and the other components of $\Lambda^1$ and $\Lambda^2$ can be paired to have the same Thurston-Bennequin invariant and rotation number. After Legendrian isotopy, we can assume that $\Lambda_1^1=\Lambda_1^2$ and, following the proof of Proposition~7.8 in \cite{DaltonEtnyreTraynor2024}, the other components of $\Lambda^1$ sit on a convex torus $T^1$ bounding a standard neighborhood of $\Lambda^1_1$ while the components of $\Lambda^2$ sit on a convex torus $T^2$ bounding a standard neighborhood of $\Lambda^2_1$. In the case that $\tb(\Lambda_1^i)>q$, we see that $\Lambda^1$ is a stabilization of a $(q-\tb(\Lambda^1_1))$ twisted $n$-copy of $\Lambda^1_1$ and similarly for $\Lambda^2$. As twisted $n$-copies and stabilizations are well-defined, the result follows. 
For details check Lemma 5.11 in  \cite{DaltonEtnyreTraynor2024}.

Now if $\tb(\Lambda_1^i)<q$, then the components of $\Lambda^i$ destabilize to ruling curves of slope $q$ on $T^i$. Ruling curves of integer slope larger than the dividing slope are Legendrian isotopic to the core Legendrian in the neighborhood. One can see this by taking an annulus from a ruling curve to the core Legendrian knot. This annulus can be made convex, and one may easily argue that all the dividing curves run from one boundary component to the other, thus giving the Legendrian isotopy. Given this, we see that $\Lambda^1$ and $\Lambda^2$ are stabilizations of ruling curves on the boundary of a standard neighborhood of $\Lambda^1_1=\Lambda^2_1$ and hence are Legendrian isotopic. 

In the case that $\tb(\Lambda_1^i)=q$, we see that $\Lambda^i$ is a stabilization of an $n$-copy. We now consider stabilizations of the $n$-copy of a Legendrian $L_i$ with $\tb(L_i)=q$. We will show that if $i\not=j$, the result of stabilizing some of the components of $nL_i$ only positively and some only negatively will not be Legendrian isotopic to the same stabilizations of $nL_j$. More specifically, we will show that $\Lambda^i=S_{1,+}^k\circ S_{2,-}^l(2L_i)$ is not Legendrian isotopic to $\Lambda^2=S_{1,+}^k\circ S_{2,-}^l(2L_j)$, and the general result follows from this. For simplicity, we will take $i=1$ and $j=2$.

We note that $\Lambda^i$ sits on the boundary of a standard neighborhood $N_i$ of $L_i$, for $i=1,2$, since $2L_i$ does. Now, if $\Lambda^1$ is Legendrian isotopic to $\Lambda^2$ then after contact isotopy, we can assume that $\Lambda^1=\Lambda^2$. Now $\partial N_1$ is smoothly isotopic to $\partial N_2$ fixing $\Lambda^1$. Thus, by discretization of isotopy \cite{Colin97, Honda02},  there is a sequence of convex tori $T_1, \ldots, T_n$ such that $T_1=\partial N_1$, $T_n=\partial N_2$, each $T_i$ contains $\Lambda^1$, and $T_i$ is obtained from $T_{i-1}$ by a bypass attachment disjoint from $\Lambda$. 

We will inductively show that for each $i$ there are tori $T'_i$ and $T''_i$ such that $T'_i$ and $T''_i$ bound solid tori $S'_i$ and $S''_i$, respectively, such that $S''_i\subset S'_i$, $S'_i$ and $S''_i$ are standard neighborhoods of a knot isotopic to $L_1$, and $T_i\subset S'_i\setminus S''_i$. This will complete the proof as we will know that $N_2$ is a standard neighborhood of $L_1$, which implies that $L_1$ and $L_2$ are Legendrian isotopic. To see that the claim is true, we note that it is true for $i=1$ as we can take an $I$-invariant neighborhood of $T_1=\partial N_1$ and let $S'_1$ and $S''_1$ be the boundaries of this neighborhood. Now assume that the claim is true for all indices less than $i$. We first suppose that $T_i$ is obtained from $T_{i-1}$ by attaching a bypass from outside the torus that $T_{i-1}$ bounds. Clearly, we can take $S_i''$ to be $S_{i-1}''$ in this case. We claim that the slope of the dividing curves on $T_i$ is the same as the slope of the dividing curves on $T_{i-1}$. If this is true, then the dividing set on $T_i$ can only differ from that on $T_{i-1}$ by the number of dividing curves it has. Since our knot type is uniformly thick we know that the solid torus $T_i$ bounds is contained in a standard neighborhood of a Legendrian knot with maximal Thurston-Bennequin invariant. Thus, we know we can find a solid torus $S'_i$ outside of the solid torus $T_i$ bounds that has the same dividing slope as $T_i$ but only has two dividing curves. Since $S_i''$ and $S_i'$ both have two dividing curves with the same slope, the contact structure on the regions between them must be non-rotative and after possibly changing the characteristic foliation on $S'_i$ the contact structure is $I$-invariant. Thus the solid torus $S_i'$ that $T_i'$ bounds must also be a standard neighborhood of $L_1$ just as $S''_i$ is. 

We are left to see that the dividing slope on $T_i$ is the same as on $T_{i-1}$. Suppose $T_i$ has a different dividing slope. As we know how bypasses affect the slope of dividing curves, we know that the dividing slope on $T_i$ must be in $(\tb(L_1),\tb(L_1)+1]$. All the components of $\Lambda_1$ sit on the torus $T_i$, and by Lemma~\ref{forcedstabs}, each of them must have been stabilized with the same sign, but our hypothesis is that one component was stabilized only positively and the other was stabilized only negatively. Thus $T_i$ must have the same dividing slope as $T_{i-1}$. 

The argument in the case when a bypass is attached on the inside of the solid torus $T_{i-1}$ bounds is almost identical and left to the reader. 

We now consider that case when $\Lambda^1$ and $\Lambda^2$ are two Legendrian representatives of $K_{n(1,q)}$ such that $\Lambda_1$ is obtained from a $n$-copy of $L_i$ with $\tb(L_i)=q$ by stabilizing all the components positively some number of times (but not negatively) and $\Lambda_2$ is obtained by the same stabilizations of an $n$-copy of $L_j$ with $\tb(L_j)=q$. Let $t$ be the largest Thurston-Bennequin invariant of a component of $\Lambda_1$. The $n$-copy of $L_i$ sits on the boundary of a standard neighborhood $N$ of $L_i$ as Legendrian divides. Inside $N$ is a standard neighborhood $N'$ of the result $L'$ of $L_i$ after positively stabilizing until $\tb=t$. Let $L'$ be ruling curves of slope $(1,q)$ on $\partial N'$. Since $q$ is larger than the dividing slope on $N'$, the ruling curves are all isotopic to $L'$. Thus, since stabilization is well-defined, $\Lambda_1$ is isotopic to a stabilization of $L'$. The same argument shows that $\Lambda_2$ is a stabilization of ruling curves on the boundary of a standard neighborhood of its component $L''$ with the largest Thurston-Bennequin invariant. Now it is clear that $\Lambda_1$ and $\Lambda_2$ are Legendrian isotopic if and only if $L'$ and $L''$ are. 
\end{proof}

We now consider the classification of Legendrian integer lesser-sloped cables of the twist knot $K_{-2n}$. 

\begin{proof}[Proof of Theorem~\ref{integerex}]
All the statements in Theorem~\ref{integerex} follow directly from Theorem~\ref{maininteger} except for the statements about stabilizations of $m$-copies of $L^i_{(0,1)}$ where all components have been stabilized both positively and negatively. The argument is identical to the one given at the end of Theorem~\ref{maininteger}.
\end{proof}

%%%%%%%%%%%%%%%%%%%%%%%%%%%%%%%%%%%%
\section{Non-integer lesser-sloped cables}
%%%%%%%%%%%%%%%%%%%%%%%%%%%%%%%%%%%%
We will begin by proving Theorem~\ref{lesserknots} as a series of results in this section.  We begin with the maximal Thurston-Bennequin invariant representatives of non-integer lesser-sloped cable knots. The following proposition is the first part of Theorem~\ref{lesserknots}.

\begin{proposition}\label{prop1}
Let $K$ be a uniformly thick knot type and $p$ and $q$ relatively prime with $q/p<\tbb(K)$. Let $L_1, \ldots, L_k$ be the distinct Legendrian knots in $\mathcal{L}(K)$ with $\tb=\lceil q/p\rceil$. Every element in $\mathcal{L}(K_{(p,q)})$ destabilizes to $(L_i)^\pm_{(p,q)}$ for some $i$ and $\pm$.

We distinguish the $(L_i)^\pm_{(p,q)}$ as follows.
\begin{enumerate}
\item If $\rot(L_i)\not= \rot(L_j)$, then $(L_i)^\pm_{(p,q)}$ and $(L_j)^\pm_{(p,q)}$ are distinct. 
\item If $\rot(L_i)=\rot(L_j)$ but $S_\pm(L_i)\not=S_\pm(L_j)$, then $(L_i)^\pm_{(p,q)}$ and $(L_j)^\pm_{(p,q)}$ are distinct. 
\item If $\rot(L_i)=\rot(L_j)$ but Legendrian surgery in $L_i$ and $L_j$ yield distinct contact manifolds, then $(L_i)^\pm_{(p,q)}$ and $(L_j)^\pm_{(p,q)}$ are distinct. 
\item Otherwise, it is not clear if $(L_i)^\pm_{(p,q)}$ and $(L_j)^\pm_{(p,q)}$ are distinct. 
\end{enumerate} 
\end{proposition} 
\begin{proof}
Given any Legendrian $L$ in $\mathcal{L}(K_{(p,q)})$ we know that its Thurston-Bennequin invariant is $\leq pq$ by \cite{EtnyreHonda05}. We wish to show that if it is $<pq$, then it destabilizes. Since the framing of $L$ given by a torus it sits on that bounds a solid torus in the knot type of $K$ is $pq$, we see that $L$ sits on a convex torus $T$ that bounds a solid torus in the knot type of $K$. If $L$ is not a ruling curve on $T$, then we can use the bypasses on $T$ to destabilize $L$. So, after destabilizing $L$ we can assume it is a ruling curve on $T$. By uniform thickness, there is a torus $T'$ disjoint from $T$ that bounds a solid torus containing $T$ or contained in the solid torus $T$ bounds. Now we can use an annulus from a Legendrian dividing curve on $T'$ to $L$ to see that $L$ destabilizes to a standard cable of $L\in \mathcal{L}(K)$, see \cite[ Lemma~7.8]{EtnyreMinMukherjee22Pre}. 

We now consider the case that $\rot(L_i)\not= \rot(L_j)$. Lemma~3.8 in \cite{EtnyreHonda05} computes the rotation number of standard cables. In particular, we see that 
\[
\rot((L_i)^\pm_{(p,q)})=p\rot(L_i)\pm (p\tb(L_i) -q)
\]
and similarly for $(L_j)^\pm_{(p,q)}$. Given that $q/p\in(\tb(L_i)-1, \tb(L_i))$, we see that the cables have distinct rotation numbers and hence are distinct. 

For the other statements, we first note a simple lemma.
\begin{lemma}\label{different}
Suppose $K$ is a uniformly thick knot type. Suppose that $T$ and $T'$ are two convex tori bounding solid tori $S$ and $S'$ in the knot type $K$, and they both have two dividing curves of slope $q/p$. If $T$ and $T'$ are not contact isotopic (after arranging their characteristic foliations are the same), then a Legendrian divide on $T$ is not Legendrian isotopic to a Legendrian divide on $T'$. 
\end{lemma}
We will prove this lemma later, but first, see how the proposition follows. Recall that $(L_i)^\pm_{(p,q)}$ is constructed as follows: we take a standard neighborhood $N_i$ of $L_i$ and inside it, we take a standard neighborhood $N'_i$ of $S_\pm(L_i)$, in $N_i\setminus N'_i$ we can find a convex torus $T_i$ with dividing slope $q/p$ and $2$ dividing curves. The knot $(L_i)^\pm_{(p,q)}$ is a Legendrian divide on this torus. We similarly can build $(L_j)^\pm_{(p,q)}$ on a convex torus $T_j$. We will show that $(L_i)^\pm_{(p,q)}$ is not Legendrian isotopic to $(L_j)^\pm_{(p,q)}$. 

In Case~(2) we are assuming that $S_\pm(L_i)\not=S_\pm(L_j)$. By the classification of tight contact structures on the solid torus \cite{Honda00a} we know that inside of the solid torus $S_i$ that $T_i$ bounds there is a unique convex torus $T'_i$ with two dividing curves of slope $\tb(L_i)-1$ and this torus bounds a standard neighborhood $S'_i$ of $S_\pm(L_i)$. We have the analogous tori for $T_j$. If $T_i$ is contact isotopic to $T_j$ then $S_i$ will be contact isotopic to $S_j$ and hence, by the uniqueness of the $T'_i$ and $T'_j$ we see that the standard neighborhoods of $S_\pm(L_i)$ and $S_\pm(L_j)$ are contact isotopic, which in turn implies $S_\pm(L_i)$ is Legendrian isotopic to $S_\pm(L_j)$. This contradiction implies $T_i$ is not contact isotopic to $T_j$, and hence by Lemma~\ref{different}, $(L_i)^\pm_{(p,q)}$ is not Legendrian isotopic to $(L_j)^\pm_{(p,q)}$.

In Case~(3) we are assuming Legendrian surgery in $L_i$ and $L_j$ yields distinct contact manifolds. Recall that we perform Legendrian surgery on $L_i$ by removing its standard neighborhood $N_i$ and gluing in a solid torus $S$ with meridian $\tb(L_i)-1$ and supporting its unique tight contact structure. Notice that we can break $S$ into two pieces: $S'$, a solid torus with dividing slope $q/p$, and a thickened torus $A$ with boundary slopes $q/p$ and $\tb(L)$. These can be chosen so that the result of gluing $A$ to the complement of $N_i$ is the complement of the solid torus $S_i$ that $T_i'$ bounds, to do this we choose the contact structure on $A$ determined by a path in the Farey graph with all $\pm$ depending on whether we are considering $(L_i)^+_{(p,q)}$ or  $(L_i)^-_{(p,q)}$. Thus, we can alternatively perform this surgery on $L_i$ by removing the torus $S_i$ and gluing in a solid torus with meridian $\tb(L_i)-1$ and with the universally tight contact structure given by a path in the Farey graph consisting of all $\pm$ signs. If $T_i$ and $T_j$ are contact isotopic, then the complements of $S_i$ and $S_j$ are contactomorphic, and thus, gluing in the above-mentioned torus to both these complements will yield contactomorphic manifolds. Thus, Legendrian surgery on $L_i$ and $L_j$ would be contactomorphic. Since these surgeries are assumed to be distinct, the tori $T_i$ and $T_j$ are not contact isotopic and the standard cables are distinct by the lemma above. 
\end{proof}
We now prove the above lemma.
\begin{proof}[Proof of Lemma~\ref{different}]
The proof uses a discretization of isotopy argument that is almost identical to the end of the proof of Theorem~\ref{maininteger}. 
\end{proof}

We now move on to determining when stabilizations of standard lesser-sloped cables become Legendrian isotopic. 
\begin{proposition}\label{prop2}
Let $K$ be a uniformly thick knot type and $p$ and $q$ relatively prime with $q/p<\tbb(K)$. Let $L_1, \ldots, L_k$ be the distinct Legendrian knots in $\mathcal{L}(K)$ with $\tb=\lceil q/p\rceil$. Every element in $\mathcal{L}(K_{(p,q)})$ destabilizes to $(L_i)^\pm_{(p,q)}$ for some $i$ and $\pm$.

We distinguish stabilizations of the $(L_i)^\pm_{(p,q)}$ as follows.
\begin{enumerate}
\item Stabilizations of $(L_i)^+_{(p,q)}$ and $(L_i)^-_{(p,q)}$ remain distinct until the first knot has been  $-$-stabilized $|p\tb(L)-q|$ or more times, the second knot has been $+$-stabilized $|p\tb(L)-q|$ or more times, and the classical invariants match. In this case they are Legendrian isotopic.
\item If $S_+(L_i)=S_-(L_j)$, then $(L_i)^+_{(p,q)}$ and $(L_j)^-_{(p,q)}$ will have distinct rotation numbers until the first has been $+$-stabilized $|p(\tb(L)-1)-q|$ times and the second has been $-$-stabilized $|p(\tb(L)-1)-q|$ times. After each has been stabilized as indicated, they become Legendrian isotopic. 
\item If $S_\pm(L_i)=S_\pm(L_j)$, but Legendrian surgery on $L_i$ and $L_j$ yield distinct contact manifolds, then stabilizations of $(L_i)^\pm_{(p,q)}$ and $(L_j)^\pm_{(p,q)}$ remain distinct until each has been $\pm$-stabilized $|p(\tb(L)-1)-q|$ or more times, and the classical invariants are the same, in this case they are Legendrian isotopic.
\item If $S^l_\pm(L_i)$ is not Legendrian isotopic to $S^l_\pm(L_j)$ for any $l$, then any number of $\pm$-stabilizations of $(L_i)^\pm_{(p,q)}$ and $(L_j)^\pm_{(p,q)}$ remain distinct.
\end{enumerate} 
\end{proposition}
\begin{proof}
Recall in the proof of Proposition~\ref{prop1} we recalled that 
\[
\rot((L_i)^\pm_{(p,q)})=p\rot(L_i)\pm (p\tb(L_i) -q).
\]
Thus, for $(L_i)^+_{(p,q)}$ and  $(L_i)^-_{(p,q)}$ to even have the same rotation numbers (which is necessary for them to be Legendrian isotopic), one must perform the stabilizations indicated in the proposition. Now consider the $(p,q)$-cable $L$ of $K$ consisting of a ruling curve of slope $q/p$ on the boundary of a standard neighborhood of $L_i$. One may use the annulus between $(L_i)^\pm_{(p,q)}$ and $L$ to see that $L$ is obtained by stabilizing $(L_i)^\pm_{(p,q)}$ $|p\tb(L)-q|$ times (see Lemma~7.8 in \cite{EtnyreMinMukherjee22Pre}). We may do the same for  $(L_i)^-_{(p,q)}$, and as these stabilizations are the same, they must be the ones claimed in Item~(1) of the proposition. 

For Item~(2) in the proposition, we note that the ruling curves of slope $(p,q)$ on a standard neighborhood of $S_+(L_i)=S_-(L_j)$ is a stabilization of both $(L_i)^+_{(p,q)}$ and $(L_j)^-_{(p,q)}$ by an argument similar to the one above. Thus, those common stabilizations are isotopic, but fewer stabilizations will have distinct rotation numbers. One may easily compute that the claimed number of stabilizations is correct. 

For Item~(3) in the proposition, we note that if we look at the $(p,q)$-cable $L'$ consisting of a $q/p$ ruling curve on the boundary of a standard neighborhood of $S_\pm(L_i)=S_\pm(L_j)$, then $L'$ is obtained from $(L_i)^\pm_{(p,q)}$ and from $(L_j)^+_{(p,q)}$ by $|p(\tb(L)-1)-q|$ $\pm$-stabilizations (just as discussed above). Thus, the claimed stabilizations of the knots are Legendrian isotopic.  

We now assume that $(L_i)^\pm_{(p,q)}$ and $(L_j)^\pm_{(p,q)}$ have not been $\pm$-stabilized at least $|p\tb(L)-q|$ times (but could be stabilized any number of times with the opposite sign). Let these Legendrian knots be denoted $L'_i$ and $L'_j$, respectively. We now argue that $L'_i$ and $L'_j$ cannot be Legendrian isotopic. Let $\Sigma'_i$ be a convex torus containing $L'_i$ and similarly  $\Sigma'_j$ containing $L'_j$. As above, we will let $N_i$ be a standard neighborhood of $L_i$ and $N'_i$ a standard neighborhood of $S_\pm(L_i)$ and similarly for $N_j$ and $N'_j$. 
We can assume that $\Sigma'_i$ is in $N_i-N'_i$ and $\Sigma'_j$ is in $N_j-N_J'$ (since the links originally sat on such a torus and the stabilization can be done by having the link inefficiently intersecting the dividing set). If we assume that $L'_i$ and $L'_j$ are Legendrian isotopic, then after a global contact isotopy, we can assume that $L'_i=L'_j$. Note there is a smooth isotopy from $\Sigma'_i$ to $\Sigma'_j$ that is fixed on $L_i'$. Thus, by discretization of isotopy \cite{Honda02}, we can assume there are tori $T_1=\Sigma'_i, T_2, \ldots, T_l=\Sigma'_j$ such that all the $T_k$ contain $L_i'$ and $T_k$ is obtained from $T_{k-1}$ by a bypass attachment. We will inductively prove that for all $k$, the torus $T_k$ bounds a solid torus that contains a convex torus $T'_k$ that is contained in $N_i-N'_i$. Once we have shown this, we will obtain a contradiction since $T'_l$ will be in $N_i-N'_i$ by induction and inside of $N_j-N'_j$ by construction (since it will have dividing slope $(tb(L_i)-1,\tb(L_i))$ and is contained in the solid torus $\Sigma'_j=T_l$ bounds), but this would imply that Legendrian surgery on $L_i$ and $L_j$ would be contactomorphic (as we argued at the end of the proof of Proposition~\ref{prop1}). 

Notice that $T_1=\Sigma'_i$ clearly satisfies the inductive hypothesis. We now assume that $T_k$ does too. If $T_{k+1}$ is obtained from $T_k$ by attaching a bypass outside of the solid torus that $T_k$ bounds, then we can take $T'_{k+1}=T'_k$. 

If the bypass were attached from the inside of the torus, then we need to consider several cases. In the first case, we assume that the dividing slope on $T_k$ is in $(\tb(L_i)-1,\tb(L_i))$. 
In this case, we first argue that $T_{k+1}$ cannot have dividing slope $\tb(L_i)-1$. If it did, then it would have to have two dividing curves (since for the slope of dividing curves to change by a bypass attachment, there can only be two), and since the solid torus bounded by $T_k$ contains $T_k'$ and hence $N'_i$ we know that $T_{k+1}$ must be $\partial N'_i$ (since this is the only torus of slope $\tb(L_i)-1$ in the solid torus bounded by $T_k$). Thus we have $L'_i$ sitting on $\partial N'_i$, but we saw above that a ruling curve of slope $q/p$ on $\partial N'_i$ will be $L^\pm_{(p,q)}$ after $|p\tb(L)-q|$ $\pm$-stabilizations and any Legendrian curve of slope $q/p$ on this torus will be a further stabilization. Since we are assuming that $L'_i$ has not been stabilized this many times, $T_{k+1}$ cannot have slope $\tb(L_i)-1$. Since $T_k$ has slope in $(\tb(L_i)-1,\tb(L_i))$, the dividing slope of $T_{k+1}$ cannot be less than $\tb(L_i)-1$ since there must be an edge in the Farey graph from the slope of $T_k$ and the slope of $T_{k+1}$ (and the edge between $\tb(L_i)$ and $\tb(L_i)-1$ prevents the existence of such an edge). This finished our claim in this case. 

We now assume that the dividing slope $s_k$ of $T_k$ is greater than or equal to $\tb(L_i)$. If $s_k$ is not an integer, then the solid torus $S_k$ that $T_k$ bounds contains $N_i$. We note by the classification of tight contact structure on solid tori that there is a unique solid torus $S_k'$ with convex boundary having dividing slope $\lfloor s_k\rfloor$ inside of $S_k$ and $N_i$ will be inside of $S_k'$ (it might be $S_k'$). Now when we attach a bypass to $T_k$ from inside $S_k$ to get $T_{k+1}$, the new dividing slope must be greater than or equal to $\lfloor s_k\rfloor$ and thus $T_{k+1}$ also bounds a solid torus containing $N_i$, and thus there is an obvious choice for $T'_{k+1}$. 

We are left to consider the case when $s_k\geq \tb(L_i)$ is an integer. If $T_k$ has more than two dividing curves, then a bypass attachment does not change the slope of the dividing curves, so we can take $T'_{k+1}=T'_k$. If $T_k$ has two dividing curves, then the solid torus $S_k$ it bounds is a standard neighborhood of a Legendrian knot $L'$. Now, $S_k$ can contain solid tori that do not contain $N_i$. To see this we first notice that $L_i=S^m_\pm(L')$, for some $m$, since if not, then one could check that the ruling curve on $T_k$ could not be obtained from $(L_i)^\pm_{(p,q)}$ by stabilization and hence any curve on $T_k$ would be a further stabilization of the ruling curve and could therefore not be $L_i'$. Indeed, notice that a ruling curve on $T_k$ would have to be a $|p(\tb(L_i)+m)-q|$ stabilization of $(L_i)^\pm_{(p,q)}$, but the rotation number of $(L_i)^\pm_{(p,q)}$ is $p\rot(L_i)\pm (p\tb(L_i)-q)$ and the rotation number of a ruling curve would be $p\rot(L')$. If $L_i=S^m_\pm(L')$, then $p\rot(L')=p(\rot(L_i)\mp m)$ which can be obtained from $p\rot(L_i)\pm (p\tb(L_i)-q)$ by $|p(\tb(L_i)+m)-q|$ stabilizations, but if $L_i$ is any other stabilization of $L'$ this  is not possible. (We note that this argument is similar to the claim that negative torus knots are Legendrian simple in \cite{EtnyreHonda01b}) Going back to the claim that $S_k$ contains solid tori that do not contain $N_i$ we note that the standard neighborhood of $S_\mp(L')$ is in $S_k$ and does not contain $N_i$. So it is possible that when we attach a bypass to $T_k$ to obtain $T_{k+1}$ that $T_{k+1}$ does not contain $N_i$, but the argument above says that if the dividing slope of $T_{k+1}$ is integral then this cannot happen. A slightly more involved computation similar to the one above implies this also cannot happen if the dividing slope is non-integral. Thus, we can choose $T_{k+1}'$ as desired. 

The proof of Item~(4) is very similar to the proof of Item~(3) and is left to the reader.
\end{proof}
Our main theorem about lesser-sloped cable knots follows from the above propositions. 
\begin{proof}[Proof of Theorem~\ref{lesserknots}]
The theorem is exactly the union of statements in Proposition~\ref{prop1} and Proposition~\ref{prop2}.
\end{proof}

We now consider applying Theorem~\ref{lesserknots} to classify non-integer lesser-sloped cable knots of negative twist knots. Specifically, we justify Figure~\ref{fig:cablemr} as the mountain range of $(p,q)$-cables of the twist knot $T_{-2n}$ when $-q/p<0$. 

\begin{proof}[Proof of Theorem~\ref{knotcable}]
Let $K$ be the $-2n$-twist knot, $-q/p\in(-m-1,-m)$ for $m\geq0$ an integer, and $k=\lceil\frac n2\rceil$. A complete, non-overlapping list of Legendrian knots in $\mathcal{L}(K)$ with $\tb=\lceil q/p \rceil=-m$ is given in \cite{EtnyreNgVertesi13} and discussed in the introduction. Specifically, it consists of 
\[
L^l_1,\ldots, L^l_k, L_1,\ldots, L_m, L^r_1,\ldots, L^r_k
\]
where $\rot(L^l_i)=-m-1$, $\rot(L^r_i)=m+1$, and $\rot(L_i)=-m-1 + 2i$ for $i=1,\ldots, m$.

From Proposition~\ref{prop1} we know that all non-destabilizable representatives of $K_{(p,-q)}$ come as standard cables of these Legendrian knots, and each of the above Legendrian knots gives us two standard $(p,-q)$-cables:
\[
(L^l_1)^\pm_{(p,-q)},\ldots, (L^l_k)^\pm_{(p,-q)}, (L_1)^\pm_{(p,-q)},\ldots, (L_m)^\pm_{(p,-q)}, (L^r_1)^\pm_{(p,-q)},\ldots, (L^r_k)^\pm_{(p,-q)}
\]
Since these are standard, lesser-sloped cables, we know that $\tb=-pq$. 
From the formula for the rotation number of cables given in the proof of Proposition~\ref{prop1} we see that
\begin{align*}
\rot&((L^l_j)^\pm_{(p,-q)}) = -p(m+1) \pm (q-p), \\
\rot& ((L_j)^\pm_{(p,-q)})= p(m+1-2i) \pm (q-p), \text{ and } \\
\rot&((L^r_1)^\pm_{(p,-q)}) = p(m+1) \pm (q-p),\\
\end{align*}
for $j=1, \ldots, k$, and $i=1,\ldots, m$. We know that the $(L^l_j)^\pm_{(p,-q)}$ are distinct by Proposition~\ref{prop1} since \cite{WanZhou24pre} tells us that Legendrian surgery on the $L^l_j$ yield distinct contact manifolds. We note that there is a contactomorphism $\phi$ of the standard tight contact structure $\ker(dz-y\, dx)$ on $\R^3$ (and hence $S^3$) of the form $(x,y,z)\mapsto (-x,-y,z)$. One may easily check that $\phi(L^r_j)$ is taken to $\phi(L^l_j)$ with its reversed orientation (of course, this is with the correct indexing on the knots). Thus, the contact structures on the complement of $L^r_j$ and $L^l_j$ are contactomorphic after reversing the orientation on one of the plane fields. This shows that Legendrian surgery on $L^r_j$ and $L^l_j$ are contactomorphic after reversing the orientation on one of the plane fields. Thus Legendrian surgeries on all the $L^r_j$ yield different contact structures and hence Proposition~\ref{prop1} again tells us that all the $(L^r_j)^\pm_{(p,-q)}$ are distinct. 

Items~(1) and~(2) in Proposition~\ref{prop2} show that the stabilizations of the $(L_j)^\pm_{(p,-q)}$ become Legendrian isotopic if and only if they have the same invariants. Items~(1) in Proposition~\ref{prop2} show that the stabilizations of the $(L^l_j)^+_{(p,-q)}$ and $(L^l_j)^-_{(p,-q)}$ become Legendrian isotopic after the first has been $-$-stabilized $|-pm+q|$ times and the second has been $+$-stabilized $|-pm+q|$ times, but remain distinct until those stabilizations have been performed; and similarly for $(L^r_j)^\pm_{(p,-q)}$. Item~(3) in that proposition says that the result of $+$-stabilizing $(L^l_j)^+_{(p,-q)}$, $|q-p(m+1)|$ and $-$-stabilizing $(L^l_j)^-_{(p,-q)}$, $|q-p(m+1)|$ times are Legendrian isotopic; and similarly for $(L^r_j)^-_{(p,-q)}$ and $(L^r_j)^+_{(p,-q)}$. Finally, Item~(4) says there are no relations among the stabilizations of the $(L^l_j)^\pm_{(p,-q)}$ other than those noted above (or further stabilizations of these); and similarly for the $(L^r_j)^\pm_{(p,-q)}$. 
\end{proof}
We now turn to cables with slope in the range $(0,1)$.
\begin{proof}[Proof of Theorem~\ref{ifsurg}]
The proof of this theorem is almost identical to the previous theorem, given the hypothesis that Legendrian surgery on the maximal Thurston-Bennequin invariant Legendrian representatives of $K$ yields distinct contact manifolds.
\end{proof}

We now move to cable links and prove Theorem~\ref{negative}
\begin{proof}[Proof of Theorem~\ref{negative}]
Recall that we are given a uniformly thick knot type $K$ and $p$ and $q$ relatively prime with $q/p<\tbb(K)$. Let $L_1, \ldots, L_k$ be the distinct Legendrian knots in $\mathcal{L}(K)$ with $\tb=\lceil q/p\rceil$. The fact that every element in $\mathcal{L}(K_{n(p,q)})$ destabilizes to $(L_i)^\pm_{n(p,q)}$ for some $i$ and $\pm$ follows from the proof of Proposition~\ref{prop1}. The fact that the various standard cable links are distinct follows from the statement for cable knots in Proposition~\ref{prop1}.

The fact that various stabilizations of the links stay distinct follows from Proposition~\ref{prop2} and that they become Legendrian isotopic follows from the proof of Proposition~\ref {prop2}.
\end{proof}
We finally turn to the classification of cable links of twist knots with an even number of negative twists. 
\begin{proof}[Theorem~\ref{lesserlink}]
This immediately follows from Theorems~\ref{knotcable} and~\ref{negative}.
\end{proof}

%%%%%%%%%%%%%%%%%%%%%%%%%%%%%%%%%%%%
\section{Uniform thickness of twist knots} \label{sec:uniform}
%%%%%%%%%%%%%%%%%%%%%%%%%%%%%%%%%%%%

In this section, we will prove Theorem~\ref{thm:uniform}, the uniform thickness of the twist knots $K_n$ for $n \notin [-3,2]$. We will only consider negative twist knots with $n < -3$ as the theorem was shown for positive twist knots with $n > 2$ by George and Myers in \cite{George13, GeorgeMyers20}. 

Now we review a decomposition of the complement of $K_n$, used to classify Legendrian twist knots in \cite{EtnyreNgVertesi13}. Let $K_n$ be a smooth negative twist knot in the standard contact $S^3$. Consider an embedded sphere in $S^3$ that intersects $K_n$ at four points as shown in Figure~\ref{fig:sphere}. 
\begin{figure}[htb]{
\begin{overpic}%[grid,tics=10] 
{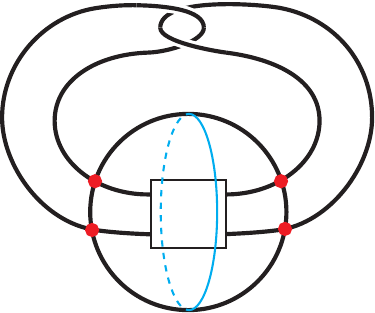}
\put(87, 45){$n$}
\put(105, 13){\color{lblue}$U$}
\end{overpic}}
\caption{A sphere that intersects $K_n$ at four points.} 
\label{fig:sphere}
\end{figure}
Now, take a solid torus $N$ in the knot type of $K_n$ with convex boundary, which could have any even number of dividing curves with any rational slope. Then, we obtain the complement of $K_n$ by removing the neighborhood from $S^3$. Now the sphere decomposes the complement of $K_n$ into two genus $2$ handlebodies $H_1$ and $H_2$. We denote the intersection of the sphere and the boundary of the handlebodies, a $4$-punctured sphere, by $S$. The handlebodies $H_1$ and $H_2$ have compressing disks. We indicate the boundary of a compressing disk, $D_i$, for $H_i$ in Figure~\ref{fig:handlebodies}.% shows the boundary of properly embedded disks in $H_1$ and $H_2$, respectively.

\begin{figure}[htb]{
\begin{overpic}%[grid,tics=10] 
{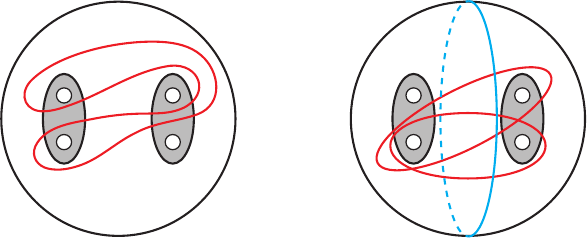}
\put(58, 98){$\partial D_1$}
\put(240, 92){\color{lblue}$U$}
\end{overpic}}
\caption{The $4$-punctured sphere $S$. The boundary, $\partial D_1$, of the disk $D_1$ in $H_1$ is shown on the left. If $n$ is even, then $\partial D_2$ is the image of the horizontal curve on the right after $\frac n2$ Dehn twists along the curve $U$. If $n$ is odd, then $\partial D_2$ is the image of the diagonal curve on the right after $\frac{n+1}2$ Dehn twist along the curve $U$. The shaded regions are $P_i$, $i=1,2$, and the outer boundary component of $P_i$ is $U_i$.} 
\label{fig:handlebodies}
\end{figure}

We can assume that the solid torus $N$ has ruling curves with meridional slope, and so we can take $S$ to have Legendrian boundary and perturb it to be convex. Let $U$ be the curve on $S$ that separates the punctures of $S$ as shown in Figure~\ref{fig:sphere}. We can assume this $U$ is Legendrian. The following observation will be a key to our proof of Theorem~\ref{thm:uniform}. 

\begin{lemma} \label{lem:unknot}
  With the notation above, if $\tb(U) = -1$, then the neighborhood $N$ of $K_n$ thickens so that the boundary has two dividing curves with an integral slope.
\end{lemma}

The proof of this lemma will be given at the end of this section, as it uses some notation established below. We now turn to the proof of our main theorem.

\begin{proof}[Proof of Theorem~\ref{thm:uniform}]
With the notation above, we assume that the neighborhood $N$ of $K_n$ has been thickened as much as possible. In addition, we assume that $U$ has been destabilized as much as possible while remaining in the complement of $N$. 

If $\tb(U)=-1$, then the previous lemma says that $N$ can be thickened to a solid torus with convex boundary having two integral sloped dividing curves, and any such solid torus is a standard neighborhood of a Legendrian representative of $K_n$. According to the classification of Legendrian twist knots \cite{EtnyreNgVertesi13}, this Legendrian knot destabilizes and hence the neighborhood thickens to a standard neighborhood of a Legendrian representative with the maximal Thurston--Bennequin invariant, thus proving that $K_n$ is uniformly thick. 

We will show that if $U$ does not have $\tb=-1$ or $N$ has not been thickened to a standard neighborhood of a Legendrian representative with the maximal Thurston--Bennequin invariant, then we will be able to find a bypass on $S$ to either destabilize $U$ or thicken $N$, but from our assumptions in the first paragraph of the proof, this is not possible, thus proving $K_n$ is uniformly thick. 

To achieve this goal, we first normalize the dividing set on $S$. Let $P_1$ and $P_2$ be pairs of pants in $S$ such that the boundary of $P_i$ is a union of two boundary components of $S$, we call these punctures, and a closed curve $U_i$ that is isotopic to $U$ for $i=1,2$, respectively. See the shaded regions in Figure~\ref{fig:handlebodies}. We can further assume that $U_1$ and $U_2$ are Legendrian and both are Legendrian isotopic to $U$ in $S$. Now suppose $\tb(U) = -m<-1$, since otherwise we would be done as noted above. Then there are $2m$ intersection points between each $U_i$ and the dividing curves in $S$. We can also assume that there are no boundary-parallel dividing curves in $S$; since otherwise, we could use this to find a bypass and thicken the neighborhood of $K_n$. 

Now there are three possible types of dividing sets in each $P_i$: null type, horizontal type, and vertical type. See Figure~\ref{fig:dividing2}.
\begin{figure}[htb]{
\begin{overpic}%[grid,tics=10] 
{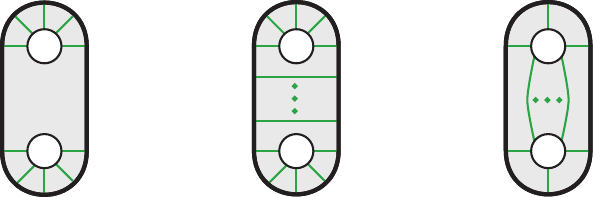}
\end{overpic}}
\caption{Three types of dividing sets in $P_1$ and $P_2$. Left: null type; middle: horizontal type; right: vertical type.} 
\label{fig:dividing2}
\end{figure}
The null type occurs when all dividing curves in $P_i$ run from the punctures to $U_i$; the horizontal type occurs when some dividing curves in $P_i$ run from $U_i$ to itself, separating two punctures; and the vertical type occurs when some dividing curves in $P_i$ run from one puncture to another.  
Since $U_1$ and $U_2$ are Legendrian isotopic, all dividing curves in $S \setminus (P_1 \cup P_2)$ run from $U_1$ to $U_2$. We also note that $P_1$ and $P_2$ should have the same type of dividing set, since otherwise, $U_1$ and $U_2$ cannot be isotopic.

We can label the intersection points of the dividing curves with $U_1$ and $U_2$, and measure the slope of the dividing curves in $S \setminus (P_1 \cup P_2)$ to be $s = \frac{k}{2m}$ for $k \in \mathbb{Z}$. See Figure~\ref{fig:dividing1} for an example. 
\begin{figure}[htb]{
\begin{overpic}%[grid,tics=10] 
{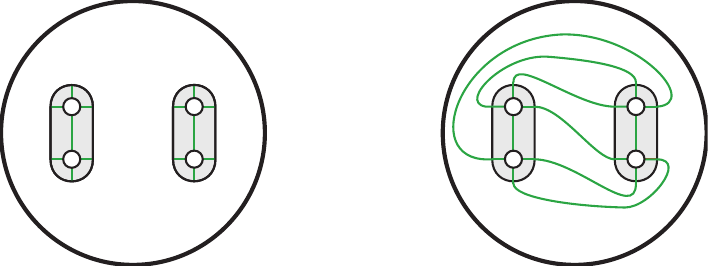}
\put(48, 74){\tiny $1$}
\put(33, 90){\tiny $2$}
\put(18, 74){\tiny $3$}
\put(18, 50){\tiny $4$}
\put(33, 33){\tiny $5$}
\put(48, 50){\tiny $6$}

\put(106, 74){\tiny $3$}
\put(91, 90){\tiny $2$}
\put(76, 74){\tiny $1$}
\put(76, 50){\tiny $6$}
\put(91, 33){\tiny $5$}
\put(106, 50){\tiny $4$}
\end{overpic}}
\caption{Left: dividing set in $P_1$ and $P_2$. Right: dividing set in $S \setminus (P_1 \cup P_2)$ with slope $s=\frac16$. The slope is positive since it is twisted in a right-handed way along $U$.} 
\label{fig:dividing1}
\end{figure}
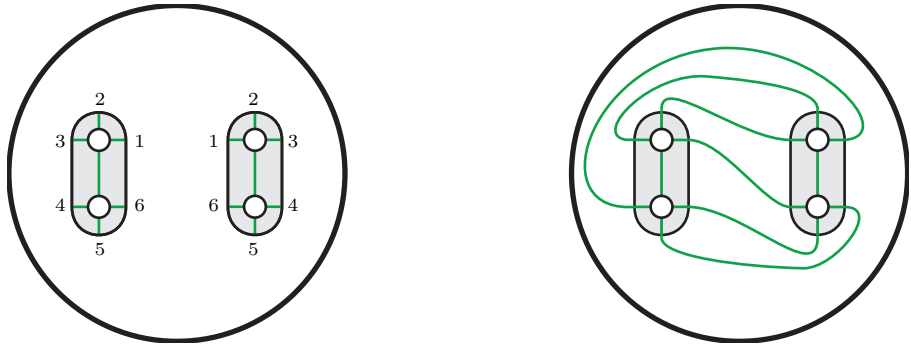
We now use the properly embedded disks $D_1$ and $D_2$, shown in Figure~\ref{fig:handlebodies}, to find a bypass that can be attached to $S$. We will use this bypass to modify the dividing set on $S$. According to the disk imbalance principle \cite[Proposition~3.18]{Honda00a}, there exists a bypass in a disk if its boundary intersects the dividing set on $S$ in at least four points.  

The following lemma guarantees that there exists a bypass in either $D_1$ or $D_2$. Recall that $s$ is the slope of the dividing curves in $S \setminus (P_1 \cup P_2)$. 

\begin{lemma} \label{lem:bypass}
  Suppose $n \leq -4$ and $\tb(U)= m \leq - 2$. If $s \leq -1$, then there exists a bypass in $D_1$ that can be attached to $S$ from the front. If $s > -1$, then there exists a bypass in $D_2$ that can be attached to $S$ from the back.
\end{lemma}

\begin{proof}
  First, since $m \leq -2$, there are at least four dividing curves in $S \setminus (P_1 \cup P_2)$. Assume $s \leq -1$. Then the purple arc $\gamma_1$ in the first drawing of Figure~\ref{fig:dividing5} has slope $0$, so it intersects the dividing curves in at least $|2m\lceil s\rceil | \geq 4$ points, since each diving curve can be obtained from a curve with slope between $-1$ and $1$ by adding $\lceil s\rceil$ Dehn twists and each Dehn twist will force the curve to intersect a zero-sloped curve an extra time. Thus, by the disk imbalance principle, there exists a bypass in $D_1$ that can be attached to $S$ from the front. 

  \begin{figure}[htb]{
  \begin{overpic}%[grid,tics=10] 
  {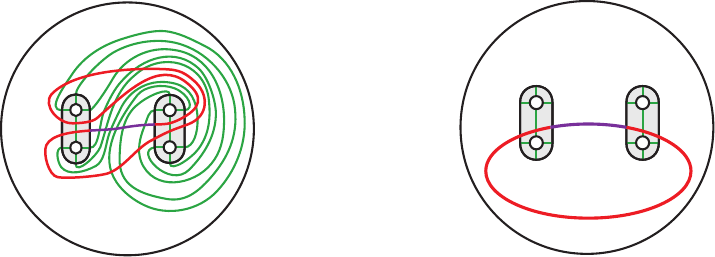}
  \put(65, 67){\tiny $\gamma_1$}
  \put(278, 69){\tiny $\gamma_2$}
  \end{overpic}}
  \caption{Left: A dividing set on $S$ of the vertical type with $s=-1$ and $\gamma_1$ is the purple arc in $\partial D_1$ that intersects the dividing curves in more than four points. Right: $\gamma_2$ is the purple arc in $\partial D_2$.} 
  \label{fig:dividing5}
  \end{figure}

  Now, assume $s > -1$ and $n$ is even. Then the purple arc $\gamma_2$ in the second drawing of Figure~\ref{fig:dividing5} has slope $\frac n2 \leq -2$, so, arguing as above, it intersects the dividing curves in at least $|2m \lceil \frac n2 -  s\rceil|\geq 4$ points. Thus, there exists a bypass in $D_2$ that can be attached to $S$ from the back. A similar argument also works when $n$ is odd. 
\end{proof}

Recall that if we attach a bypass from the front, then the dividing set is modified in a specific way (see \cite{Honda00a}), and when the bypass is attached from the back there is also a specific modification that is the mirror of the one for attaching a bypass to the front of a surface. 

We now investigate possible bypass attachments for $S$. We first observe that if the attaching arc of a bypass lies in $S \setminus (P_1 \cup P_2)$, then the bypass attachment results in a destabilization of $U$. Moreover, according to the bypass sliding lemma \cite[Lemma 1.3]{HondaKazezMatic03}, if the arcs can be slid outside of $P_1 \cup P_2$, we can find an actual bypass outside of $P_1 \cup P_2$ and destabilize $U$. Figure~\ref{fig:dividing3} lists the attaching arcs for bypass in $P_1$ and $P_2$ that cannot be slid out. In the following three lemmas, we examine these arcs and show that they can either be used to destabilize $U$ or to thicken the neighborhood of $K_n$, thereby completing the proof of Theorem~\ref{thm:uniform}. Recall that we can assume that $P_1$ and $P_2$ have the same type of dividing set; otherwise, $U$ can be destabilized. We begin with the case that $P_1$ and $P_2$ are of null type. 

\begin{figure}
\centering
    \begin{subfigure}[b]{0.16\textwidth}
    \centering
    \includegraphics[width=1cm]{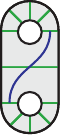}
    \caption{\label{fig:bypassA}}
    \end{subfigure}
\quad
    \begin{subfigure}[b]{0.16\textwidth}
    \centering
    \includegraphics[width=1cm]{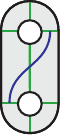}
    \caption{\label{fig:bypassB}}
    \end{subfigure}
\quad
    \begin{subfigure}[b]{0.16\textwidth}
    \centering
    \includegraphics[width=1cm]{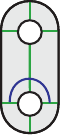}
    \caption{\label{fig:bypassC}}
    \end{subfigure}
\quad
    \begin{subfigure}[b]{0.16\textwidth}
    \centering
    \includegraphics[width=1cm]{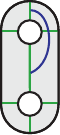}
    \caption{\label{fig:bypassD}}
    \end{subfigure}
\quad
    \begin{subfigure}[b]{0.16\textwidth}
    \centering
    \includegraphics[width=1cm]{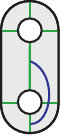}
    \caption{\label{fig:bypassE}}
    \end{subfigure}

\vspace{0.3cm}

    \begin{subfigure}[b]{0.16\textwidth}
    \centering
    \includegraphics[width=1cm]{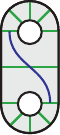}
    \caption{\label{fig:bypassF}}
    \end{subfigure}
\quad
    \begin{subfigure}[b]{0.16\textwidth}
    \centering
    \includegraphics[width=1cm]{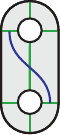}
    \caption{\label{fig:bypassG}}
    \end{subfigure}
\quad
    \begin{subfigure}[b]{0.16\textwidth}
    \centering
    \includegraphics[width=1cm]{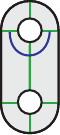}
    \caption{\label{fig:bypassH}}
    \end{subfigure}
\quad
    \begin{subfigure}[b]{0.16\textwidth}
    \centering
    \includegraphics[width=1cm]{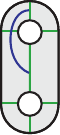}
    \caption{\label{fig:bypassI}}
    \end{subfigure}
\quad
    \begin{subfigure}[b]{0.16\textwidth}
    \centering
    \includegraphics[width=1cm]{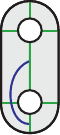}
    \caption{\label{fig:bypassJ}}
    \end{subfigure}
\caption{The list of attaching arcs in $P_1$ and $P_2$ that cannot be slid outside of $P_1 \cup P_2$ when $P_1$ and $P_2$ are either horizontal or vertical.}
\label{fig:dividing3}
\end{figure}

\begin{lemma}\label{lem:null}
  If $P_1$ and $P_2$ are of null type, then $U$ can be destabilized or $N$ can be thickened. 
\end{lemma}

\begin{proof}
The only attaching arc for a bypass that cannot be slid outside of $P_1\cup P_2$, and hence gives a destabilization of $U$, is a bypass that intersects two of the arcs near one puncture and one arc near the other puncture. But when pushing $S$ past this bypass, we will see that there is a bypass on $S$ for either $\partial N$ or $U$.
\end{proof}

We now turn to the case when $P_1$ and $P_2$ are horizontal. 

\begin{lemma}\label{lem:horizontal}
  If $P_1$ and $P_2$ are of horizontal type, then one of the following holds:
  \begin{enumerate}
    \item $U$ can be destabilized, or 
    \item the ambient contact structure on $S^3$ is overtwisted.
  \end{enumerate}
\end{lemma}

\begin{proof}
It is straightforward to verify that any attaching arc can be slid outside of $P_1 \cup P_2$, except for \subref{fig:bypassA} and \subref{fig:bypassF} in Figure~\ref{fig:dividing3}. We first assume that $s \leq -1$. Then, according to Lemma~\ref{lem:bypass}, we can find a bypass contained in $D_1$, so the bypass is attached to the front of $S$. In this case, attaching \subref{fig:bypassA} destabilizes $U$. Also, attaching \subref{fig:bypassF} increases the slope of dividing curves in $S\setminus (P_1\cup P_2)$ by $\frac{1}{2m}$. Therefore, any bypass except for \subref{fig:bypassF} will destabilize $U$, so we assume we can only find \subref{fig:bypassF}. Then we can keep increasing the slope of dividing curves in $S \setminus (P_1\cup P_2)$ and obtain an integral slope. Recall that $S$ was a subset of a sphere in $S^3$ and after we cap off the punctures of $S$ (by disks in the neighborhood $N$), there is more than one dividing curve on the sphere, which implies that the ambient contact structure in $S^3$ is overtwisted and contradicts that the knot lies in the standard contact $S^3$. See Figure~\ref{fig:dividing4}. 

\begin{figure}[htb]{
\begin{overpic}%[grid,tics=10] 
{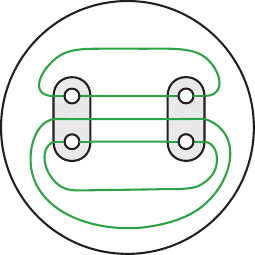}
\end{overpic}}
\caption{A dividing set on $S$ of the horizontal type with $s=0$.} 
\label{fig:dividing4}
\end{figure}  
 
If $s > -1$, then we can use Lemma~\ref{lem:bypass} to find a bypass contained in $D_2$, and attaching it results in a bypass attached from the back of $S$. Attaching \subref{fig:bypassA} decreases the slope of the dividing curves in $S\setminus(P_1\cup P_2)$ by $\frac{1}{2m}$ and attaching \subref{fig:bypassF} destabilizes $U$. Thus, an argument similar to the one above also works here. 
\end{proof}

We now turn to the case that $P_1$ and $P_2$ are vertical.

\begin{lemma}\label{lem:vertical}
  If $P_1$ and $P_2$ are of vertical type, then one of the following holds:
  \begin{enumerate}
    \item $U$ can be destabilized, or 
    \item there exists a bypass that thickens the neighborhood of $K_n$. 
  \end{enumerate}
\end{lemma}

\begin{proof}
  It is straightforward to verify that if there is more than one vertical component, then all possible bypass attachments destabilize $U$ or yield a boundary-parallel dividing curve that can be used to thicken the neighborhood of $K_n$. Thus we assume that there is a single vertical component in each $P_i$. In this case, it is routine to check that every attaching arc can be slid outside of $P_1 \cup P_2$ except for the ones in Figure~\ref{fig:dividing3}. 
  
  Assume $s \leq -1$. Then by Lemma~\ref{lem:bypass}, we can find a bypass contained in $D_1$. In this case, bypasses are attached from the front of $S$, and \subref{fig:bypassB} is the only attaching arc that does not yield a boundary-parallel dividing curve. Attaching \subref{fig:bypassB} decreases the slope of the dividing curves in $S \setminus (P_1 \cup P_2)$ by $\frac{1}{2m}$. Notice that changing the slope does not yield an overtwisted contact structure as in the previous proof. However, when $s \leq -1$, the only attaching arcs that can be found in $\partial D_1$ are \subref{fig:bypassE}, \subref{fig:bypassG}, \subref{fig:bypassI} and the ones that can be slid outside of $P_1 \cup P_2$; in particular, we cannot get a bypass of type (b) along $\partial D_1$. See Figure~\ref{fig:dividing5} for an example when $s=-1$. Thus we can destabilize $U$ or thicken the neighborhood of $K_n$ after attaching a bypass in $D_1$. 

  If $s > -1$, then by Lemma~\ref{lem:bypass}, we can find a bypass contained in $D_2$. In this case, the bypass is attached to the back of $S$, and \subref{fig:bypassG} is the only attaching arc that does not yield a boundary-parallel dividing curve. Attaching \subref{fig:bypassG} increases the slope of the dividing curves in $S \setminus (P_1\cup P_2)$ by $\frac{1}{2m}$. However, when $s > -1$, the only attaching arcs that can be found in $\partial D_2$ are \subref{fig:bypassB}, \subref{fig:bypassD}, \subref{fig:bypassJ} and the ones that can be slid outside of $P_1 \cup P_2$. Thus, we can destabilize $U$ or thicken the neighborhood of $K_n$ after attaching a bypass in $D_2$. 
\end{proof}
So we see in all cases we can destabilize $U$ or thicken $N$, leading to a contradiction unless $\tb(U)$ is $-1$ and $N$ is a neighborhood of a maximal Thurston-Bennequin invariant representative of $K_n$. 
\end{proof}

We now prove the key lemma used above.
\begin{proof}[Proof of Lemma~\ref{lem:unknot}]
If $\tb(U)=-1$, then the dividing set on $P_i$ intersects $U_i$ exactly twice. Thus, if there is not a bypass on $P_i$ that allows one to thicken the neighborhood $N$ the dividing set on $P_i$ must have one dividing curve running from $U_i$ to one of the punctures, one dividing curve running from $U_i$ to the other puncture and all other dividing curves must run between the punctures (this is a simple case of $P_i$ being of vertical type).  

If $N$ does not have an integral dividing slope, then there will be more than $2$ dividing curves running between the punctures of $P_i$. Thus, both $D_1$ and $D_2$ will have at least two boundary parallel dividing curves giving bypasses for $S$. One may easily check that attaching a bypass to $S$ along all the bypass attaching arcs on $\partial D_2$ will result in a dividing curve on $S$ that will give a bypass for $N$ and hence can be used to thicken $N$. (We note that since the bypasses on $D_1$ are attached to the front of $S$, it is not true that such a bypass exists on $D_1$.) Thus, we may thicken $N$ until it has two dividing curves with integer slope. In this case, any bypasses found for $S$ along $D_1$ or $D_2$ can only affect the slope of the dividing curves on $S-(P_1\cup P_2)$. 
\end{proof}

\def\cprime{$'$}

% references
%\bibliography{references}
%\bibliographystyle{plain}
\end{document}